\documentclass[11pt,a4paper]{amsart}
\usepackage{geometry}
\usepackage[colorlinks, linkcolor=red, anchorcolor=blue, citecolor=blue]{hyperref}
\usepackage{color}             
\usepackage{amsthm,bm} 
\usepackage{enumerate}
\usepackage{amsfonts}
\usepackage{amsmath}
\usepackage{mathrsfs}
\usepackage{amssymb}
\usepackage{amsbsy}
\usepackage{ifthen}
\usepackage[all]{xy}
\usepackage{extarrows}
\usepackage{indentfirst}        
\usepackage{latexsym}        

\usepackage{graphicx}
\usepackage{cases}
\usepackage{pifont}
\usepackage{txfonts}
\usepackage{xcolor}
\usepackage{multirow}
\usepackage{caption, subcaption}
\usepackage{tikz}

\newcounter{maint}

\def\k{\Bbbk}
\def\I{\mathbb{I}}
\def\ufo{\mathfrak{ufo}}
\newcommand{\Dchaintwo}[4]{
\rule[-3\unitlength]{0pt}{8\unitlength}
\begin{picture}(14,5)(0,3)
\put(2,4){\ifthenelse{\equal{#1}{l}}{\circle*{4}}{\circle{4}}}
\put(4,4){\line(2,0){20}}
\put(26,4){\ifthenelse{\equal{#1}{r}}{\circle*{4}}{\circle{4}}}
\put(2,10){\makebox[0pt]{\scriptsize #2}}
\put(14,8){\makebox[0pt]{\scriptsize #3}}
\put(26,10){\makebox[0pt]{\scriptsize #4}}
\end{picture}}

\usepackage[vcentermath]{youngtab}

\usepackage{listings}
\numberwithin{equation}{section}

\newcommand{\ydh}{{}^H_H\mathcal{YD}}

\allowdisplaybreaks[4] 

\begin{document}

\newtheorem{theorem}{Theorem}[section]

\newtheorem{lemma}[theorem]{Lemma}

\newtheorem{corollary}[theorem]{Corollary}
\newtheorem{proposition}[theorem]{Proposition}

\theoremstyle{remark}
\newtheorem{remark}[theorem]{Remark}

\theoremstyle{definition}
\newtheorem{definition}[theorem]{Definition}

\theoremstyle{definition}
\newtheorem{conjecture}[theorem]{Conjecture}

\newtheorem{example}[theorem]{Example}
\newtheorem{problem}[theorem]{Problem}
\newtheorem{question}[theorem]{Question}


\def\k{\Bbbk}
\def\id{\mathrm{id}}
\def\ad{\mathrm{ad}}
\title[Finite-dimensional Nichols algebras over the Suzuki algebras  
\uppercase\expandafter{\romannumeral3}]
{Finite-dimensional Nichols algebras over the Suzuki algebras  
\uppercase\expandafter{\romannumeral3}:  simple Yetter-Drinfeld modules}

\author[Shi]{Yuxing Shi }
\address{Jiangxi Provincial Center for Applied Mathematics, and School of Mathematics and Statistics, Jiangxi Normal University, Nanchang, Jiangxi 330022, People's Republic of China}
\email{yxshi@jxnu.edu.cn}

\makeatletter
\@namedef{subjclassname@2020}{\textup{2020} Mathematics Subject Classification}
\makeatother

\begin{abstract}
In this paper, we continue to  investigate  finite-dimensional 
Nichols algebras over simple Yetter-Drinfeld modules of the Suzuki algebras $A_{N\, n}^{\mu\lambda}$. 
It is finished for the case $A_{N\, 2n}^{\mu\lambda}$. As for the case  $A_{N\, 2n+1}^{\mu\lambda}$, 
it boils down to the long-standing open problem: calculate dimensions of Nichols algebras of dihedral rack type $\Bbb D_{2n+1}$. 
It is interesting to see that  the Suzuki algebras are of set-theoretical type. 
We pose some question or problems for our future research. In particular, we are curious about how to generalize 
the correspondence between braidings of rack type and  group algebras to braidings and Hopf algebras of 
set-theoretical type.
\end{abstract}
\maketitle
\tableofcontents

\section{Introduction}
Let $\k$ be  an algebraicaly  closed field of characteristic $0$ and $\k^\times$ be $\k-\{0\}$. 
In this paper, we continue to  study 
Nichols algebras over simple Yetter-Drinfeld modules of the Suzuki algebras $A_{Nn}^{\mu\lambda}$. 
\subsection{Main results of the paper}
In \cite{Shi2020even} and \cite{Shi2020odd}, we investigated Nichols algebras over simple Yetter-Drinfeld modules of the Suzuki algebras, left the following cases unsolved. 
\begin{problem}
Determine dimensions of the following Nichols algebras: 
\begin{enumerate}
\item $\mathfrak{B}\left(\mathscr{K}_{jk,p}^{s}\right)$ in cases $n\geq 3$ or $n=2$, $\lambda=1$, 
see \cite{Shi2020even};
\item $\mathfrak{B}\left(\mathscr{L}_{k,pq}^s\right)$,  see \cite{Shi2020odd};
\item $\mathfrak{B}(V_{abe})$, $b^2\neq ae\neq 1$, $b^2\neq (ae)^{-1}$, 
          $b\in\Bbb{G}_{m}$ for $m\geq 3$, see  \cite{Shi2020even, Shi2020odd};
\item $\mathfrak{B}\left(\mathscr{N}_{k,pq}^s\right)$,  see \cite{Shi2020odd}.
\end{enumerate}
\end{problem}
In this paper, dimensions of Nichols algebras 
$\mathfrak{B}\left(\mathscr{K}_{jk,p}^{s}\right)$
and $\mathfrak{B}(V_{abe})$ are completely determined, see Theorem \ref{NicholsAlgK} and Theorem \ref{fomulaeVabe} respectively.  The Nichols algebras $\mathfrak{B}\left(\mathscr{L}_{k,pq}^s\right)$ are of 
dihedral rack type $\Bbb D_{2n+1}$, see Lemma \ref{NicholsAlgL}. The Yetter-Drinfeld modules 
$\mathscr{N}_{k,pq}^s$ are t-equivalent with  some braided vector spaces  of 
dihedral rack type $\Bbb D_{2n+1}$, see Theorem \ref{NicholsAlgN}.  If two braided vector spaces 
are t-equivalent, then dimensions of their corresponding Nichols algebras are the same, see 
Remark \ref{Remark_T_Equivalent}. 

As a summary, Nichols algebras over simple Yetter-Drinfeld modules of $A_{N\, 2n}^{\mu\lambda}$
are completely settled. 
\begin{theorem}\label{MainTheorem}
Let $M$ be a simple Yetter-Drinfeld module over $A_{N\,2n}^{\mu\lambda}$. If 
 $\mathfrak{B}(M)$ is finite-dimensional, then $\mathfrak{B}(M)$
can be summarized  as follows. 
\begin{enumerate}
\item Diagonal type \cite[Theorem 1.2]{Shi2020even}:
\begin{enumerate}
\item Cartan type $A_1$;
\item Cartan type $A_1\times A_1$;
\item Cartan type $A_2$;
\item  Cartan type $A_2\times A_2$;      
\item Super type ${\bf A}_{2}(q;\I_2)$;
\item The Nichols algebra $\ufo(8)$.
\end{enumerate}
\item Non-diagonal type:
\begin{enumerate}
\item Dihedral rack type $\Bbb D_4$, $64$-dimensional, see \cite[Lemma 4.22]{Shi2020even};
\item The Nichols algebra $\mathfrak{B}(V_{abe})$, 
\begin{align*}
\dim\mathfrak{B}(V_{abe})
=\left\{\begin{array}{ll}
4m, &b=-1, ae\in\Bbb{G}_m,  m>2, \\
m^2, &ae=1, b\in\Bbb{G}_m, m> 2,\\
\end{array}\right.
\end{align*}
see Theorem \ref{fomulaeVabe}, the Nichols algebras $\mathfrak{B}\left(\mathscr{G}_{jk,p}^{st}\right)$ and $\mathfrak{B}\left(\mathscr{H}_{jk,p}^{s}\right)$ in \cite[Section 4.2]{Shi2020even}.
\end{enumerate}
\end{enumerate}
\end{theorem}
\begin{remark}
The Kac-Paljutkin algebra $H_8$ \cite{MR0208401} is isomorphic to $A_{12}^{+-}$ as Hopf algebras.
All finite-dimensional Nichols algebras over Yetter-Drinfeld modules of $A_{12}^{+-}$
were obtained in \cite{Shi2019}. 
\end{remark}
As for  $A_{N\,2n+1}^{\mu\lambda}$, unsolved cases $\mathfrak{B}\left(\mathscr{L}_{k,pq}^s\right)$
and $\mathfrak{B}\left(\mathscr{N}_{k,pq}^s\right)$ are summarized to the long-standing open problem:
\begin{problem}
Determine dimensions of Nichols algebras of dihedral rack type $\Bbb D_{2n+1}$. 
\end{problem}
\begin{remark}
Nichols algebras of dihedral rack type $\Bbb D_{2n}$  $(n>2)$ are infinite dimensional, according to 
\cite[Theorem 3.6]{Andruskiewitsch2011} and Example \ref{TypeD}.
\end{remark}

\subsection{Future research}
There are interesting connections between the Suzuki algebras and the dihedral groups. 
We summarize them as follows. 
\begin{enumerate}
\item The Suzuki algebras $A_{Nn}^{\mu\lambda}$ are abelian extensions by 
the dihedral goup $D_{2n}$, see formula \eqref{AbelianExtension}.
\item The  dimension  distribution of simple Yetter-Drinfeld 
modules over $A_{Nn}^{\mu\lambda}$ is $1$, $2$ and $n$, which is the same 
with the dihedral goup $D_{2n}$, see \cite[Theorem 3.1]{Shi2020even}, \cite[Theorem 3.1]{Shi2020odd} 
and \cite{andruskiewitsch2007pointed}. 
\item Those $n$-dimensional simple Yetter-Drinfeld modules over $A_{Nn}^{\mu\lambda}$
are either of dihedral rack type $\Bbb D_{n}$, or  t-equivalent with braided vector spaces 
of dihedral rack type $\Bbb D_{n}$, see Theorem \ref{NicholsAlgK}, Theorem \ref{NicholsAlgN}, 
Lemma \ref{NicholsAlgI} and Lemma \ref{NicholsAlgL}. 
\end{enumerate}
It would be interesting to find a theoretical interpretation for these connections. It is well-known that 
braidings of rack type can be realized in categories of  Yetter-Drinfeld modules over group algebras  and 
Yetter-Drinfeld modules over any finite group algebras  are of rack type \cite{Andruskiewitsch2003MR1994219}. 
Given any non-degenerate set-theoretical solution of the Yang-Baxter equation $(X, r)$, we can 
construct braided vector spaces $W_{X, r}$, see Definition \ref{BraidedVectorSpace}. 
We call braided vector spaces $W_{X, r}$ are of set-theoretical type.  It is obvious that 
set-theoretical type includes rack type. Let $H$ be a Hopf algebra. If all Yetter-Drinfeld modules over 
$H$ are of set-theoretical type, we say the Hopf algebra $H$ is of set-theoretical type. 
The suzuki algebras are of set-theoretical type, see Lemma \ref{LemmaSetTheoretical}. 
We want to ask how to generalize the correspondence 
between braidings of rack type and  group algebras to braidings and Hopf algebras of 
set-theoretical type.

\begin{problem}
Realize braidings of set-theoretical type in categories of  Yetter-Drinfeld modules. 
\end{problem}

\begin{problem}
Classify Hopf algebras  of set-theoretical type. 
\end{problem}

\begin{question}
Let $H$ be any finite-dimensional Hopf algebra arising from abelian extension. Is $H$ of set-theoretical type?
\end{question}

\begin{remark}
Let $\Sigma=FG$ be an exact factorization of a finite group $\Sigma$. 
If a Hopf algebra $H$ is fitting into an abelian extension
\[
1\to \k^G\to H\to \k F\to 1, 
\]
associated to this factorization. 
Natale \cite{Natale2003} proved  that $H$ is group theoretical and the 
Drinfeld double $D(H)$ is a twisting of the quasi-Hopf algebra $D^\psi(\Sigma)$ \cite{Dijkgraaf1991} 
for  some 3-cocycle $\psi\in Z^3(\Sigma, \k^\times)$. From \cite[Formula 2.9]{Huang2018}, 
it seems braidings in 
${}_{\k \Sigma}^{\k \Sigma}\mathcal{YD}^{\psi}$ 
are of rack type in case that $\Sigma$ is abelian. Maybe these informations can connect with the fact 
that braidings of set-theoretical type could be conjugated with braidings of Rack type, 
see Example \ref{SetRackRelations}.  
\end{remark}

\begin{problem}
Classify finite-dimensional Nichols  algebras  of set-theoretical type. 
\end{problem}
\section{Preliminaries}
\subsection{The suzuki algebras $A_{Nn}^{\mu\lambda}$}
Suzuki introduced a family of cosemisimple Hopf algebras $A_{Nn}^{\mu\lambda}$ 
which is parametrized by integers $N\geq 1$, $n\geq 2$ and $\mu$, $\lambda=\pm 1$ \cite{Suzuki1998}. 
The Suzuki algebras $A_{Nn}^{\mu\lambda}$ are obtained by abelian extension
\begin{align}\label{AbelianExtension}
1 \to H\to A_{Nn}^{\mu\lambda}\to \k D_{2n}\to 1, 
\end{align}
where $H=\k \langle h_{+}, h_-\rangle$, $h_{\pm}$ are group-likes of  $A_{Nn}^{\mu\lambda}$
and $D_{2n}$ is the dihedral group of order $2n$, see \cite[Page 18]{Suzuki1998}.
The Suzuki Hopf algebra  $A_{Nn}^{\mu\lambda}$ is generated by $x_{11}$, $x_{12}$, $x_{21}$, 
$x_{22}$ subject to the relations:
\begin{align*}
&x_{11}^2=x_{22}^2,\quad x_{12}^2=x_{21}^2,\quad \chi _{21}^n=\lambda\chi _{12}^n,
\quad \chi _{11}^n=\chi _{22}^n,\\
&x_{11}^{2N}+\mu x_{12}^{2N}=1,\quad
x_{ij}x_{kl}=0\,\, \text{whenever $i+j+k+l$ is odd}, 
\end{align*}
where $\chi _{11}^m$, $\chi _{12}^m$, $\chi _{21}^m$ and $\chi _{22}^m$ are
defined as follows for  $m\in \Bbb Z^+$: 
$$\chi _{11}^m:=\overbrace{x_{11}x_{22}x_{11}\ldots\ldots }^{\textrm{$m$ }},\quad \chi _{22}^m:=\overbrace{x_{22}x_{11}x_{22}\ldots\ldots }^{\textrm{$m$ }},$$
$$\chi _{12}^m:=\overbrace{x_{12}x_{21}x_{12}\ldots\ldots }^{\textrm{$m$ }},\quad \chi _{21}^m:=\overbrace{x_{21}x_{12}x_{21}\ldots\ldots }^{\textrm{$m$ }}.$$
The comultiplication, counit and antipode  of $A_{Nn}^{\mu\lambda}$ are given by
\begin{equation}\label{eq5.3}
\Delta (\chi_{ij}^k)=\chi_{i1}^k\otimes \chi_{1j}^k+\chi_{i2}^k\otimes \chi_{2j}^k,\quad 
\varepsilon(x_{ij})=\delta_{ij}, \quad S(x_{ij})=x_{ji}^{4N-1}, 
\end{equation}
for $k\geq 1$, $i,j=1,2$. Let $\overline{i,i+j}=\{i,i+1,i+2,\cdots,i+j\}$ be  an index set, 
then the  basis of $A_{Nn}^{\mu\lambda}$ can be represented by 
\begin{equation}\label{eq5.2}
\left\{x_{11}^s\chi _{22}^t,\ x_{12}^s\chi _{21}^t \mid 
s\in\overline{1,2N}, t\in\overline{0,n-1} 
\right\}.
\end{equation}

\subsection{Nichols algebras}
\begin{definition}\cite[Definition 2.1]{andruskiewitsch2001pointed}
\label{defNicholsalgebra}
Let $H$ be a Hopf algebra and $V \in \ydh$. A braided $\mathbb{N}$-graded 
Hopf algebra $R = \bigoplus_{n\geq 0} R(n) \in \ydh$  is called 
the \textit{Nichols algebra} of $V$ if 
\begin{enumerate}
 \item[(i)] $\k \cong R(0)$, $V\cong R(1) \in \ydh$.
 \item[(ii)] $R(1) = \mathcal{P}(R)
=\{r\in R~|~\Delta_{R}(r)=r\otimes 1 + 1\otimes r\}$.
 \item[(iii)] $R$ is generated as an algebra by $R(1)$.
\end{enumerate}
In this case, $R$ is denoted by $\mathfrak{B}(V) = \bigoplus_{n\geq 0} \mathfrak{B}^{n}(V) $.    
\end{definition}
\begin{remark}
Let $(V, c)$ be a braided vector space, then 
the Nichols algebra 
$\mathfrak{B}(V)$ is completely determined by the braiding $c$.
More precisely, as proved in  \cite{MR1396857} and
noted in \cite{andruskiewitsch2001pointed},
$$\mathfrak{B}(V)=\k\oplus V\oplus\bigoplus\limits_{n=2}^\infty V^{\otimes n} /
\ker\mathfrak{S}_n=T(V)/\ker\mathfrak{S},$$
where $\mathfrak{S}_{n,1}\in \mathrm{End}_\k \left(V^{\otimes (n+1)}\right)$, 
$\mathfrak{S}_{n}\in \mathrm{End}_\k \left(V^{\otimes n}\right)$,
\begin{align*}
c_i&\coloneqq \mathrm{id}^{\otimes (i-1)}\otimes c
\otimes \mathrm{id}^{\otimes (n-i-1)}\in \mathrm{End}_\k \left(V^{\otimes n}\right), \\
\mathfrak{S}_{n,1} &\coloneqq\mathrm{id}+c_n+c_{n-1}c_n+\cdots
+c_1\cdots c_{n-1}c_n=\mathrm{id}+\mathfrak{S}_{n-1,1}c_n,\\
\mathfrak{S}_1&\coloneqq\mathrm{id}, \quad \mathfrak{S}_2\coloneqq\mathrm{id}+c, \quad
\mathfrak{S}_n\coloneqq \mathfrak{S}_{n-1,1}(\mathfrak{S}_{n-1}\otimes \mathrm{id}).
\end{align*}
\end{remark}

\begin{definition}\cite[Definition 5.10]{Andruskiewitsch2003MR1994219}
\label{Remark_T_Equivalent}
Two braided vector spaces $(V, c)$ and $(W, \tilde c)$ are t-equivalent if there is a collection of 
linear isomorphisms $U^n: V^{\otimes n}\to W^{\otimes n}$ intertwining the corresponding representations 
of the braid group $\Bbb B_n$ for all $n\geq 2$. 
\end{definition}
\begin{remark}\cite[Lemma 6.1]{Andruskiewitsch2003MR1994219}
If braided vector spaces $(V, c)$ and $(W, \tilde c)$ are t-equivalent, then 
the corresponding Nichols algebras  $\mathfrak{B}(V)$ and $\mathfrak{B}(W)$ are isomorphic as 
graded vector spaces. Hence $\dim \mathfrak{B}(V)=\dim \mathfrak{B}(W)$.
\end{remark}

\begin{lemma}\label{MainLemma}
Let $(V, c)$ be a braided vector space, and $\varphi_1, \varphi_2: V\to V$ be two invertible linear maps. 
If 
\[
\tilde c=(\varphi_1^{-1}\otimes \varphi_2^{-1})c(\varphi_1\otimes \varphi_2)
=(\varphi_2^{-1}\otimes \varphi_1^{-1})c(\varphi_2\otimes \varphi_1),
\]
 then  $(V, \tilde c)$ is a braided vector space and is t-equivalent to $(V, c)$.
\end{lemma}

\begin{proof}
Let $T: V^{\otimes n}\to V^{\otimes n}$ be an invertible map with 
$T=T_1\otimes T_2\otimes \cdots \otimes T_n$ such that $T_i$ equals $\varphi_1$ in case 
that $i$ is odd and $\varphi_2$ in case that $i$ is even.
If $i$ is odd, then 
\begin{align*}
\tilde c_i&={\rm id}^{\otimes (i-1)}\otimes \tilde c\otimes {\rm id}^{\otimes (n-i-1)}\\
&=T_1^{-1}T_1\otimes \cdots  \otimes T_{i-1}^{-1}T_{i-1}\otimes \left((f_1^{-1}\otimes f_2^{-1})c
       (f_1\otimes f_2)\right)\otimes 
       T_{i+2}^{-1}T_{i+2}\otimes \cdots\otimes T_n^{-1}T_n\\
&=T^{-1}c_iT.       
\end{align*}
If $i$ is even,  then 
\begin{align*}
\tilde c_i&={\rm id}^{\otimes (i-1)}\otimes \tilde c\otimes {\rm id}^{\otimes (n-i-1)}\\
&=T_1^{-1}T_1\otimes \cdots  \otimes T_{i-1}^{-1}T_{i-1}\otimes \left((f_2^{-1}\otimes f_1^{-1})c
       (f_2\otimes f_1)\right)\otimes 
       T_{i+2}^{-1}T_{i+2}\otimes \cdots\otimes T_n^{-1}T_n\\
&=T^{-1}c_iT.    
\end{align*}
We have $\tilde c_{i+1}\tilde c_i \tilde c_{i+1}=T^{-1}c_{i+1}c_ic_{i+1}T=T^{-1}c_ic_{i+1}c_iT
=\tilde c_{i}\tilde c_{i+1} \tilde c_{i}$. So $(V, \tilde c)$ is a braided vector space and is t-equivalent to $(V, c)$.
\end{proof}

\subsection{Set-theoretical solution of the Yang-Baxter equation}
A set-theoretical solution of the Yang-Baxter equation is a  pair $(X, r)$, where 
$X$   is a non-empty set and $r: X\times X\rightarrow X\times X$ is a bijective map
such that
\[
(r\times {\rm id})({\rm id}\times r)(r\times {\rm id})
=({\rm id}\times r)(r\times {\rm id})({\rm id}\times r)
\]
holds. By convention, we write 
\[
r(i, j)=(\sigma_i(j), \tau_j(i)),\quad  \forall i, j\in X.
\]
A solution $(X, r)$ is non-degenerate if all the maps $\sigma_i: X\rightarrow X$
and $\tau_i: X\rightarrow X$ are bijective for all $i\in X$, and involutive if 
$r^2=\id_{X\times X}$.

\begin{definition}\label{BraidedVectorSpace}
\cite[Lemma 5.7]{Andruskiewitsch2003MR1994219}
Let $(X, r)$ be a non-degenerate set-theoretical  solution of the Yang-Baxter equation, $|X|=m\in\Bbb Z^{\geq 2}$.
Then $W_{X, r}=\bigoplus_{i\in X}\k w_i$ is a braided vector 
space with the braiding given by 
\begin{align}\label{YBEquation}
c(w_{i}\otimes w_j)
&=R_{i,j} w_{\sigma_i(j)}\otimes w_{\tau_j(i)}, \quad 
\text{where}\,\, R_{i,j}\in\k^\times\,\,\text{and} \\ 
R_{i, j}R_{\tau_j(i), k}R_{\sigma_i(j),\sigma_{\tau_j(i)}(k)}
&=R_{j,k}R_{i,\sigma_j(k)}R_{\tau_{\sigma_j(k)}(i),\tau_k(j)}, 
\quad \forall i, j, k\in X. \label{YBEquation}
\end{align}
The braided vector space  $W_{X, r}$ is called of set-theoretical type. 
\end{definition}

\begin{remark}
The braiding of $W_{X,r}$ is rigid according to 
\cite[Lemma 3.1.3]{Schauenburg1992}, which means that $W_{X,r}$ can be realized in categories of 
Yetter-Drinfeld modules.
\end{remark}

If all Yetter-Drinfeld modules over a Hopf algebra $H$ are of set-theoretical type, then 
we say $H$ is  of set-theoretical type. 

\begin{lemma}\label{LemmaSetTheoretical}
The Suzuki algebra $A_{Nn}^{\mu\lambda}$ is  of set-theoretical type. 
\end{lemma}
\begin{proof}
Let $M=\bigoplus_{i\in I}\k w_i$ be any simple Yetter-Drinfeld module over $A_{Nn}^{\mu\lambda}$. 
According to \cite{Shi2020even, Shi2020odd},  the module structure of $(M, \cdot)$
and the  comodule structure of $(M, \rho)$ satisfies the following conditions respectively: 
\begin{enumerate}
\item For any $i\in I$, there exist some $j, k\in I$ such that 
\begin{align*}
& x_{11}\cdot w_i\in \k^{\times} w_j, \quad 
x_{22}\cdot w_i\in \k^{\times} w_k, \quad  x_{12}\cdot w_i=x_{21}\cdot w_i=0, \\
\text{or\quad} & x_{12}\cdot w_i\in \k^{\times} w_j, \quad 
x_{21}\cdot w_i\in \k^{\times} w_k, \quad x_{11}\cdot w_i=x_{22}\cdot w_i=0; 
\end{align*}
\item For any $i\in I$, there exist some 
$\alpha\in\left\{x_{11}^s\chi _{22}^t, \mid s\in\overline{1,2N}, t\in\overline{0,n-1} 
\right\}$, $\beta\in \left\{x_{12}^s\chi _{21}^t \mid s\in\overline{1,2N}, t\in\overline{0,n-1} 
\right\}$ and some $j, k\in I$ such that 
\begin{align*}
\rho(w_i)\in \k^{\times} \alpha\otimes w_j+\k^{\times} \beta \otimes w_k.
\end{align*}
\end{enumerate}
Since the category of Yetter-Drinfeld modules over $A_{Nn}^{\mu\lambda}$ is semisimple, it 
is easy to see that $A_{Nn}^{\mu\lambda}$ is of set-theoretical type. 
\end{proof}

\subsection{Racks}
\begin{definition}
Let $X$ be a non-empty set, then $(X,\rhd)$ is a rack if $\rhd: X\times X\to X$ is a function, such that 
$\phi_i: X\to X$, $\phi_i(j)=i\rhd j$,  is bijection for all $i\in X$ and 
\begin{align}
i\rhd(j\rhd k)=(i\rhd j)\rhd (i\rhd k), \quad \forall i, j, k\in X.
\end{align}
\end{definition}

\begin{remark}\cite{Brieskorn1988}
$(X, \rhd)$ is a rack if and only if $(X, r)$ is a set-theoretical solution of the Yang-Baxter equation, where 
$r(x, y)=(x\rhd y, x)$ for all $x, y\in X$. 
\end{remark}

\begin{example}\cite{zbMATH01585085, LuJianghua2000MR1769723}
\label{SetRackRelations}
Let $(X, r)$ be a non-degenerate set-theoretical solution of the Yang-Baxter equation. Then 
$(X, \rhd)$ is a rack with $x\rhd y=\tau_x\sigma_{\tau_y^{-1}(x)}(y)$ for all $x, y\in X$. 
Let $T: X\times X\rightarrow X\times X$ with $T(x, y)=(\tau_y(x), y)$, then 
$T$ is invertible and $T^{-1}(x, y)=(\tau_y^{-1}(x), y)$. We have 
\begin{align}
TrT^{-1}(x, y)=(x\rhd y, x).
\end{align}
\end{example}

\begin{example}\label{TypeD}
Let $\Bbb D_n=(\Bbb Z_n, \rhd)$, $i\rhd j=2i-j\in\Bbb Z_n$ for all $i, j\in \Bbb Z_n$. 
Then $\Bbb D_n$ is a rack and called a dihedral rack.  If $n>2$, then the dihedral rack 
$\Bbb D_{2n}$ is of type $D$, see \cite[Lemma 2.2]{Andruskiewitsch2010}. 
\end{example}

\section{The Nichols algebras over the Suzuki algebras}
\subsection{The Nichols algebras $\mathfrak{B}(V_{abe})$}
Let  $V_{abe}=\k v_1\oplus \k v_2$ be a  braided vector space with  three parameters $a, b, e\in \k$, and the braiding given by: 
\begin{align*}
c(v_1\otimes v_1)&=a v_2\otimes v_2,\quad 
&c(v_1\otimes v_2)&=b   v_1\otimes v_2,\\
c(v_2\otimes v_1)&=b v_2\otimes v_1,\quad 
&c(v_2\otimes v_2)&=e   v_1\otimes v_1.
\end{align*}
\begin{theorem}\cite[Proposition 3.16]{Andruskiewitsch2018}
\begin{align}\label{fomulaeVabe}
\dim\mathfrak{B}(V_{abe})
=\left\{\begin{array}{ll}
27, & ae=b^2, b^3=1\neq b, \quad \text{Cartan type $A_2$},\\
4m, &b=-1, ae\in\Bbb{G}_m,  m\in \Bbb Z^+, \\
m^2, &ae=1, b\in\Bbb{G}_m\,\,\text{for}\,\,m\geq 2,\\
\infty, & otherwise, 
\end{array}\right.
\end{align}
where $\Bbb{G}_m$ denotes the set of $m$-th primitive roots of unity. 
\end{theorem}
\begin{remark}
Finite-dimensional Nichols algebras $\mathfrak{B}(V_{abe})$ were obtained by 
\cite[Proposition 3.16]{Andruskiewitsch2018}. In \cite{Shi2020even} and \cite{Shi2020odd}, 
the author proved that $\mathfrak{B}(V_{abe})$ is infinite dimensional in case 
$b^2=(ae)^{-1}$,  $b\in\Bbb{G}_{m}$ for   $m\geq 3$. In case $\dim\mathfrak{B}(V_{abe})=m^2$, 
those Nichols algebras are related with Pascal's triangle, see arXiv:2103.06489. This kind of 
interesting connection is generalized to Nichols algebras of dimension $n^m$ and multinomial expansion 
$(x_1+x_2+\cdots+x_m)^n$, see \cite{Shi2023}. 
\end{remark}
\begin{proof}
Let $w_1=\sqrt[4]{\frac ea} v_1$, $w_2=v_2$, then 
\begin{align*}
c(w_1\otimes w_1)&=\sqrt{ae} w_2\otimes w_2,\quad 
&c(w_1\otimes w_2)&=b   w_1\otimes w_2,\\
c(w_2\otimes w_1)&=b w_2\otimes w_1,\quad 
&c(w_2\otimes w_2)&=\sqrt{ae}   w_1\otimes w_1.
\end{align*}
So $V_{abe} \cong V_{\sqrt{ae}\,\, b\, \sqrt{ae}}$ as braided vector space. 
Let $W=V_{\sqrt{ae}\,\, b\, \sqrt{ae}}$ and 
 $\varphi: W\to W$ be an invertible 
map with $\varphi(w_1)=w_2$,   
$\varphi(w_2)=w_1$. It is easy to see that 
\[
\tilde c=(\varphi^{-1}\otimes {\rm id})c(\varphi\otimes {\rm id})
=({\rm id} \otimes \varphi^{-1})c({\rm id} \otimes \varphi), 
\]
and $(W, \tilde c)$ is a braided vector space of diagonal type.  More precisely, we have 
\begin{align*}
\tilde c(w_1\otimes w_1)&=b w_1\otimes w_1,\quad 
&\tilde c(w_1\otimes w_2)&=\sqrt{ae}   w_2\otimes w_1,\\
\tilde c(w_2\otimes w_1)&=\sqrt{ae} w_1\otimes w_2,\quad 
&\tilde c(w_2\otimes w_2)&=b   w_2\otimes w_2.
\end{align*}
According to Lemma \ref{MainLemma} or \cite[Example 5.11]{Andruskiewitsch2003MR1994219}, $V_{abe}$ is t-equivalent to $(W, \tilde c)$. 
So $\dim \mathfrak{B}(V_{abe})=\dim \mathfrak{B}(W, \tilde c)$. By Heckenberger's work \cite{heckenberger2009classification}, we have
\[
\dim \mathfrak{B}(W, \tilde c)=
\left\{\begin{array}{ll}
27, & ae=b^2, b^3=1\neq b, \quad \text{Cartan type $A_2$},\\
4m, &b=-1, ae\in\Bbb{G}_m,  m\geq 2, \quad 
\text{super type ${\bf A}_{2}(q;\I_2)$ } \text{\cite{Andruskiewitsch2017}},\\
m^2, &ae=1, b\in\Bbb{G}_m\,\,\text{for}\,\,m\geq 2, \quad \text{Cartan type $A_1\times A_1$},\\
\infty, & otherwise.
\end{array}\right.
\]
\end{proof}

\subsection{The Nichols algebras 
$\mathfrak{B}\left(\mathscr{K}_{jk,p}^{s}\right)$}\label{section_Nichos_K}

Let $\mathscr{K}_{jk,p}^s=\bigoplus_{a=1}^{2n}\k w_a$ be the $2n$-dimensional  simple Yetter-Drinfeld module over 
$A_{N\, 2n}^{\mu\lambda}$, where   $j=\left\{\begin{array}{ll}
          1\,\text{or}\,3, &\text{if}\,\lambda=-1,\\
          2\,\text{or}\,4, &\text{if}\,\lambda=1,
          \end{array}\right.
          $
          $k\in\overline{0,N-1}$, $p\in \Bbb{Z}_2$,
          $s\in\overline{1,N}$, 
see   \cite[Appendix]{Shi2020even}.  Set 
\begin{align*}
b+2a-2=2nr+d,\quad r\in\Bbb{N},\quad d\in\overline{0, 2n-1},\\
2n+1-b+2a-2=2ne+f, \quad e\in\Bbb{N},\quad f\in\overline{0, 2n-1},
\end{align*}
then the braiding of $\mathscr{K}_{jk,p}^{s}$ can be described as follows \cite[Section 4.4]{Shi2020even}:  
\begin{align*}
&\quad c(w_a\otimes w_b)\\
&=\left\{\begin{array}{ll}
(-1)^p\left(\bar{\mu}K\right)^s w_b\otimes w_1, & a=1,\\
(-1)^p\lambda^r (\bar{\mu}K)^{s+n(r-2)}J^{n(r-2)}
w_{2n}\otimes w_{2n-a+2}, &a>1, d=0, 2\mid (a+b),\\
(-1)^p\lambda^{r+1} (\bar{\mu}K)^{s+n(r-1)}J^{n(r-1)} 
w_{d}\otimes w_{2n-a+2}, &a>1, d>0, 2\mid (a+b),\\
(-1)^p\lambda^e\left(\bar{\mu}K\right)^{s-ne-2+2a}J^{-ne}
w_{1}\otimes w_{2n-a+2}, &a>1, f=0, 2\nmid (a+b),\\
\frac{(-1)^p\lambda^{e+1}\left(\bar{\mu}K\right)^{s-n(e+1)-2+2a}}{J^{n(e+1)}}
w_{2n+1-f}\otimes w_{2n-a+2}, &a>1, f>0, 2\nmid (a+b), \\
\end{array}\right. 
\end{align*}
where
$K=\omega^{8nk}$,
$J=\omega^{2jN}$.
The parameter $\omega$ is a primitive $8nN$-th root of unity, so $J^{4n}=1$.

\begin{theorem}\cite[Lemma 4.23]{Shi2020even}\label{NicholsAlgK}
The Yetter-Drinfeld module $\mathscr{K}_{jk,p}^{s}$ is t-equivalent to a braided vector space 
associated to dihedral rack $\Bbb D_{2n}$. Let $q=(-1)^p \left(K \bar{\mu}\right)^{s}$. 
\begin{enumerate}
\item If $n=1$, then 
\[
\dim \mathfrak{B}\left(\mathscr{K}_{jk,p}^{s}\right)<\infty\iff
\left\{\begin{array}{ll}
\lambda q^2=1\neq q, & \text{Cartan type}\,\, A_1\times A_1,\\
\lambda q^3=1\neq q, & \text{Cartan type}\,\, A_2.
\end{array}\right.
\]
\item If $n=2$, then 
\[
\dim \mathfrak{B}\left(\mathscr{K}_{jk,p}^{s}\right)=
\left\{\begin{array}{ll}
64, &q=-1 \text{ and } \lambda=1, \text{Cartan type $A_2\times A_2$}, \\
\infty, & otherwise. 
\end{array}\right.
\]
\item If $n>2$, then $\dim \mathfrak{B}\left(\mathscr{K}_{jk,p}^{s}\right)=\infty$.
\end{enumerate}
\end{theorem}
\begin{remark}
In case $n\leq 2$, those finite-dimensional Nichols algebras are already obtained in
\cite[Lemma 4.23]{Shi2020even}.  The relations of the $64$-dimensional Nichols 
algebra of Cartan type $A_2\times A_2$ is presented in \cite[Remark 4.24]{Shi2020even}.
\end{remark}
\begin{proof}
Set $x_{2k}=x_1 (J^{-1}K\bar{\mu})^{n-2k+1}$ for $k\in\overline{1, n}$ 
and $x_{2k+1}=\lambda x_1 (JK\bar{\mu})^{-n+2k}$ for $k\in\overline{1, n-1}$. 
It is a case by case verification for $\tilde c=\bar c$, see Lemmas \ref{NicholsEvenK_tildeC}
and \ref{NicholsEvenK_barC}. 
For example, if $2\mid a$, $2\mid b$, $2n+2a-b=2nr+d$, $r\in\Bbb N$, 
$d\in\overline{0, 2n-1}$, then 
\begin{align*}
4n+2+b-2a=\left\{\begin{array}{ll}
2n(3-r)+2-d, & d=0 \text{ or } d=2, \\
2n(2-r)+2n+2-d, & d\in\overline{4, 2n-2}.
\end{array}\right.
\end{align*}
\begin{enumerate}
\item If $d=0$, then $2n+2a-b=2nr$ and 
\begin{align*}
\tilde c(w_a\otimes w_b)&=\dfrac{x_a(-1)^p\lambda^r (\bar{\mu}K)^{s+n(r-2)}J^{n(r-2)}}
          {x_{2n}}w_{2n}\otimes w_a,\\
\bar c(w_a\otimes w_b)&=\dfrac{x_b(-1)^p\lambda^{3-r+1} (\bar{\mu}K)^{s+n(3-r-1)}J^{n(3-r-1)}}
          {x_{a}}w_{2n}\otimes w_a.          
\end{align*}
It is easy to see that  $\tilde c(w_a\otimes w_b)=\bar c(w_a\otimes w_b)
\Leftrightarrow x_a^2 (\bar{\mu}JK)^{2n(r-2)}=x_bx_{2n}$. 
This condition holds, since 
\begin{align*}
\dfrac{x_bx_{2n}}{x_a^2}
&=\dfrac{x_1 (J^{-1}K\bar{\mu})^{n-b+1}\cdot x_1 (J^{-1}K\bar{\mu})^{n-2n+1}}
   {[x_1 (J^{-1}K\bar{\mu})^{n-a+1}]^2}
=(J^{-1}K\bar{\mu})^{2a-b-2n}=(J^{-1}K\bar{\mu})^{2n(r-2)}.
\end{align*} 
\item If $d=2$, then $2n+2a-b=2nr+2$ and 
\begin{align*}
\tilde c(w_a\otimes w_b)&=\dfrac{x_a(-1)^p\lambda^{r+1} (\bar{\mu}K)^{s+n(r-1)}J^{n(r-1)}}
          {x_{d}}w_{d}\otimes w_a,\\
\bar c(w_a\otimes w_b)&=\dfrac{x_b(-1)^p\lambda^{3-r} (\bar{\mu}K)^{s+n(3-r-2)}J^{n(3-r-2)}}
          {x_{a}}w_{2}\otimes w_a.          
\end{align*}
In this case,  $\tilde c(w_a\otimes w_b)=\bar c(w_a\otimes w_b)
\Leftrightarrow x_a^2 (\bar{\mu}JK)^{2n(r-1)}=x_bx_{2}.$
This condition holds, since 
\begin{align*}
\dfrac{x_bx_{2}}{x_a^2}
&=\dfrac{x_1 (J^{-1}K\bar{\mu})^{n-b+1}\cdot x_1 (J^{-1}K\bar{\mu})^{n-2+1}}
   {[x_1 (J^{-1}K\bar{\mu})^{n-a+1}]^2}
=(J^{-1}K\bar{\mu})^{2a-b-2}=(J^{-1}K\bar{\mu})^{2n(r-1)}.
\end{align*} 
\item If $d\in\overline{4, 2n-2}$, then 
\begin{align*}
\tilde c(w_a\otimes w_b)&=\dfrac{x_a(-1)^p\lambda^{r+1} (\bar{\mu}K)^{s+n(r-1)}J^{n(r-1)}}
          {x_{d}}w_{d}\otimes w_a,\\
\bar c(w_a\otimes w_b)&=\dfrac{x_b(-1)^p\lambda^{2-r+1} (\bar{\mu}K)^{s+n(2-r-1)}J^{n(2-r-1)}}
          {x_{a}}w_{d}\otimes w_a.          
\end{align*}
In this case,  $\tilde c(w_a\otimes w_b)=\bar c(w_a\otimes w_b)
\Leftrightarrow x_a^2 (\bar{\mu}JK)^{2n(r-1)}=x_bx_{d}.$
This condition holds, since 
\begin{align*}
\dfrac{x_bx_{d}}{x_a^2}
&=\dfrac{x_1 (J^{-1}K\bar{\mu})^{n-b+1}\cdot x_1 (J^{-1}K\bar{\mu})^{n-d+1}}
   {[x_1 (J^{-1}K\bar{\mu})^{n-a+1}]^2}
=(J^{-1}K\bar{\mu})^{2a-b-d}
=(J^{-1}K\bar{\mu})^{2n(r-1)}.
\end{align*} 

\end{enumerate}

According to Lemma \ref{MainLemma}, 
$\left(\mathscr{K}_{jk,p}^{s}, c\right)$ and 
$\left(\mathscr{K}_{jk,p}^{s}, \bar c\right)$ are t-equivalent. From  Lemma \ref{NicholsEvenK_barC}, it is easy to see that the braided vector space 
$\left(\mathscr{K}_{jk,p}^{s}, \bar c\right)$ is associated to the dihedral rack $\Bbb D_{2n}$. 
In case $n>2$, then 
\[
\dim \mathfrak{B}\left(\mathscr{K}_{jk,p}^{s}\right)
=\dim \mathfrak{B}\left(\mathscr{K}_{jk,p}^{s}, \bar c\right)=\infty, 
\]
according to Remark \ref{Remark_T_Equivalent} and \cite[Theorem 3.6]{Andruskiewitsch2011}.

In case $n=2$, the braiding of the braided vector space $\left(\mathscr{K}_{jk,p}^{s}, \bar c\right)$ 
is given by 
\begin{align*}
\bar c(w_{1}\otimes w_{1})&= q w_{1}\otimes w_{1},&
\bar c(w_{1}\otimes w_{2})&= \frac{K \bar{\mu} q}{J} w_{4}\otimes w_{1},\\
\bar c(w_{1}\otimes w_{3})&= \lambda q w_{3}\otimes w_{1},&
\bar c(w_{1}\otimes w_{4})&= \frac{J q}{K \bar{\mu}} w_{2}\otimes w_{1},\\
\bar c(w_{2}\otimes w_{1})&= \frac{\lambda q}{J^{5} K \bar{\mu}} w_{3}\otimes w_{2},&
\bar c(w_{2}\otimes w_{2})&=  q w_{2}\otimes w_{2},\\
\bar c(w_{2}\otimes w_{3})&= \frac{K \lambda \bar{\mu} q}{J^{3}} w_{1}\otimes w_{2},&
\bar c(w_{2}\otimes w_{4})&= J^{4} \lambda q w_{4}\otimes w_{2},\\
\bar c(w_{3}\otimes w_{1})&= \lambda q w_{1}\otimes w_{3},&
\bar c(w_{3}\otimes w_{2})&= \frac{K \lambda \bar{\mu} q}{J^{5}} w_{4}\otimes w_{3},\\
\bar c(w_{3}\otimes w_{3})&=  q w_{3}\otimes w_{3},&
\bar c(w_{3}\otimes w_{4})&= \frac{\lambda q}{J^{3} K \bar{\mu}} w_{2}\otimes w_{3},\\
\bar c(w_{4}\otimes w_{1})&= \frac{ q}{J^{5} K \bar{\mu}} w_{3}\otimes w_{4},&
\bar c(w_{4}\otimes w_{2})&= \frac{\lambda q}{J^{4}} w_{2}\otimes w_{4},\\
\bar c(w_{4}\otimes w_{3})&= \frac{K  \bar{\mu} q}{J^{3}} w_{1}\otimes w_{4},&
\bar c(w_{4}\otimes w_{4})&=  q w_{4}\otimes w_{4}.
\end{align*}
Both $\k w_1\oplus \k w_3$ and $\k w_2\oplus \k w_4$ are braided vector subspaces of diagonal type
and  their Dynkin diagrams are given by 
$\xymatrix{ \overset{q}{\underset{\ }{\circ}}\ar  @{-}[rr]^{q^2}  &&
\overset{q}{\underset{\ }{\circ}}}$ if $q\neq -1$. 
So we have 
\[
\dim \mathfrak{B}\left(\mathscr{K}_{jk,p}^{s}, \bar c\right)<\infty \iff 
q=-1 \text{ or } q^3=1\neq q.
\]  
According to \cite{Heckenberger2017}, if  $\dim \mathfrak{B}\left(\mathscr{K}_{jk,p}^{s}, \bar c\right)<\infty$, then $\dim \mathfrak{B}\left(\mathscr{K}_{jk,p}^{s}, \bar c\right)=64$. 
This happens only in case $J^4=1$ and $q=-1$. The condition $J^4=1$ implies that $\lambda =1$. 
It is finished by \cite[Lemma 4.23]{Shi2020even}.
\end{proof}

Let $\varphi_1, \varphi_2: \mathscr{K}_{jk,p}^{s}\to \mathscr{K}_{jk,p}^{s}$ be two invertible 
maps satisfying 
\begin{align*}
\varphi_2(w_a)&=\left\{\begin{array}{ll}
w_{2n+2-a}, & 2\mid a,\\
w_a, & 2\nmid a,
\end{array}\right.\\
\varphi_1(w_b)&=\left\{\begin{array}{ll}
x_bw_{2n+2-b}, &b\neq 1,  2\nmid b, x_b\in\k^{\times},\\
x_bw_b, & b=1\,\,\text{or}\,\, 2\mid b, x_b\in\k^{\times}.
\end{array}\right.
\end{align*}

\begin{lemma}\label{NicholsEvenK_barC}
Let $\bar c=(\varphi_2^{-1}\otimes \varphi_1^{-1})c(\varphi_2\otimes \varphi_1)$, and  $\Psi_E=2nr+d$, $\Psi_O=2ne+f$, where $r, e\in\Bbb N$,  $d, f\in\overline{0, 2n-1}$, and 
\begin{align*}
\Psi_E&=\left\{\begin{array}{ll}
4n+2+b-2a, & 2\mid a, 2\mid b,\\
2a-1, & 2\nmid a, b=1, \\
2n+2a-b, & 2\nmid a, b\neq 1, 2\nmid b,
\end{array}\right.\\
\Psi_O&=\left\{\begin{array}{ll}
6n+2-2a, & 2\mid a, b=1,\\
4n+1+b-2a, & 2\mid a, b\neq 1, 2\nmid b, \\
2n-1+2a-b, & 2\nmid a,  2\mid b.
\end{array}\right.
\end{align*}
\begin{enumerate}
\item If $a=1$, then 
\begin{align*}
\bar c(w_a\otimes w_b)=\left\{\begin{array}{ll}
x_bx_1^{-1}(-1)^p(\bar{\mu}K)^s w_{2n+2-b}\otimes w_1, & b\,\,\text{even},\\
(-1)^p(\bar{\mu}K)^s w_1\otimes w_1, &b=1,\\
x_bx_1^{-1}(-1)^p(\bar{\mu}K)^s w_{2n+2-b}\otimes w_1, & 1\neq b\,\,\text{odd}.\\
\end{array}\right.
\end{align*}
\item If $a\neq 1$ and $2\mid (a+b)$, then 
\begin{align*}
\bar c(w_a\otimes w_b)=\left\{\begin{array}{ll}
x_bx_a^{-1}(-1)^p\lambda^r(\bar{\mu}K)^{s+n(r-2)}J^{n(r-2)} w_2\otimes w_a, & d=0,\\
x_bx_a^{-1}(-1)^p\lambda^{r+1}(\bar{\mu}K)^{s+n(r-1)}J^{n(r-1)} 
     w_{2n+2-d}\otimes w_a, & 0\neq d\,\,\text{even},\\
x_bx_a^{-1}(-1)^p\lambda^{r+1}(\bar{\mu}K)^{s+n(r-1)}J^{n(r-1)} 
     w_{d}\otimes w_a, & d\,\,\text{odd}.\\
\end{array}\right.
\end{align*}
\item If $a\neq 1$ and $2\nmid (a+b)$, then 
\begin{align*}
\bar c(w_a\otimes w_b)=\left\{\begin{array}{ll}
x_bx_a^{-1}(-1)^p\lambda^e(\bar{\mu}K)^{s+n(4-e)+2-2a}J^{-ne} w_1\otimes w_a, & f=0,\vspace{1mm}\\
\dfrac{x_b(-1)^p\lambda^{e+1}(\bar{\mu}K)^{s+n(3-e)+2-2a}}
       {x_aJ^{n(e+1)}} 
       w_{2n+1-f}\otimes w_a, & 0\neq f\,\,\text{even},\vspace{1mm}\\
\dfrac{x_b(-1)^p\lambda^{e+1}(\bar{\mu}K)^{s-n(e+1)-2+2a}}
       {x_aJ^{n(e+1)}}
       w_{1+f}\otimes w_a, & f\,\,\text{odd}.\\
\end{array}\right.
\end{align*}
\end{enumerate}
\end{lemma}

\begin{proof}
If $2\mid a$ and $2\mid b$, set 
\[
b+2(2n+2-a)-2=4n+2+b-2a=2nr+d, \quad r\in\{0, 1, 2\}, d\in\overline{0, 2n-1}.
\]
We have  
\begin{align*}
\bar c(w_a\otimes w_b)&=x_b (\varphi_2^{-1}\otimes \varphi_1^{-1})c(w_{2n+2-a}\otimes w_b)\\
&=\left\{\begin{array}{ll}
(-1)^p\lambda^r (\bar{\mu}K)^{s+n(r-2)}J^{n(r-2)}
    x_b (\varphi_2^{-1}\otimes \varphi_1^{-1})(w_{2n}\otimes w_a),& d=0,\\
(-1)^p\lambda^{r+1} (\bar{\mu}K)^{s+n(r-2)}J^{n(r-1)} 
   x_b (\varphi_2^{-1}\otimes \varphi_1^{-1})(w_{d}\otimes w_a),& d>0,\\
\end{array}\right.\\
&=\left\{\begin{array}{ll}
x_bx_a^{-1} (-1)^p\lambda^r (\bar{\mu}K)^{s+n(r-2)}J^{n(r-2)}
    w_{2}\otimes w_a,& d=0,\\
x_bx_a^{-1}(-1)^p\lambda^{r+1} (\bar{\mu}K)^{s+n(r-1)}J^{n(r-1)} 
   w_{2n+2-d}\otimes w_a,& d>0.\\
\end{array}\right.
\end{align*}

\par\noindent
If $2\mid a$ and $b=1$, set 
\[
2n+1-b+2(2n+2-a)-2=6n+2-2a=2ne+f, \quad e\in\{0, 1, 2\}, f\in\overline{0, 2n-1}.
\] 
We have
\begin{align*}
\bar c(w_a\otimes w_b)&=x_1 (\varphi_2^{-1}\otimes \varphi_1^{-1})c(w_{2n+2-a}\otimes w_1)\\
&=\left\{\begin{array}{ll}
\frac{x_1(-1)^p\lambda^e(\bar{\mu}K)^{s-ne-2+2(2n+2-a)}}{J^{ne}}(\varphi_2^{-1}\otimes \varphi_1^{-1})
(w_1\otimes w_a), &f=0,\vspace{1mm}\\ 
\frac{x_1(-1)^p\lambda^{e+1}(\bar{\mu}K)^{s-n(e+1)-2+2(2n+2-a)}}
{J^{n(e+1)}}(\varphi_2^{-1}\otimes \varphi_1^{-1})
(w_{2n+1-f}\otimes w_a), &f>0,\\ 
\end{array}\right.\\
&=\left\{\begin{array}{ll}
x_1x_a^{-1}(-1)^p\lambda^e(\bar{\mu}K)^{s+n(4-e)+2-2a}J^{-ne}
w_1\otimes w_a, &f=0,\\ 
\frac{x_1(-1)^p\lambda^{e+1}(\bar{\mu}K)^{s+n(3-e)+2-2a}}
{x_a J^{n(e+1)}}w_{2n+1-f}\otimes w_a, &f>0.\\ 
\end{array}\right.
\end{align*}
\par\noindent 
If $2\mid a$, $b\neq 1$ and $2\nmid b$, set
\[
2n+1-(2n+2-b)+2(2n+2-a)-2=4n+1+b-2a=2ne+f, 
\]
where $e\in\{0, 1, 2\}$, $f\in\overline{0, 2n-1}$.
\begin{align*}
\bar c(w_a\otimes w_b)&=x_b (\varphi_2^{-1}\otimes \varphi_1^{-1})c(w_{2n+2-a}\otimes w_{2n+2-b})\\
&=\left\{\begin{array}{ll}
x_b(-1)^p\lambda^e(\bar{\mu}K)^{s-ne-2+2(2n+2-a)}J^{-ne}(\varphi_2^{-1}\otimes \varphi_1^{-1})
(w_1\otimes w_a), &f=0,\\ 
\frac{x_b(-1)^p\lambda^{e+1}(\bar{\mu}K)^{s-n(e+1)-2+2(2n+2-a)}}
{J^{n(e+1)}}(\varphi_2^{-1}\otimes \varphi_1^{-1})
(w_{2n+1-f}\otimes w_a), &f>0,\\ 
\end{array}\right.\\
&=\left\{\begin{array}{ll}
x_bx_a^{-1}(-1)^p\lambda^e(\bar{\mu}K)^{s+n(4-e)+2-2a}J^{-ne}
w_1\otimes w_a, &f=0,\\ 
\frac{x_b(-1)^p\lambda^{e+1}(\bar{\mu}K)^{s+n(3-e)+2-2a}}
{x_a J^{n(e+1)}}w_{2n+1-f}\otimes w_a, &f>0.\\ 
\end{array}\right.
\end{align*}
\par\noindent
Let $2\nmid a$, $2\mid b$. If $a\neq 1$,  set 
\[
2n-1+2a-b=2ne+f, \quad e\in\{0, 1, 2\}, f\in\overline{1,2n-1}.
\]
\begin{align*}
\bar c(w_a\otimes w_b)&=x_b (\varphi_2^{-1}\otimes \varphi_1^{-1})c(w_{a}\otimes w_b)\\
&=\left\{\begin{array}{ll}
x_b(-1)^p(\bar{\mu}K)^s  (\varphi_2^{-1}\otimes \varphi_1^{-1})(w_{b}\otimes w_1), &a=1,\\
\frac{x_b(-1)^p\lambda^{e+1}(\bar{\mu}K)^{s-n(e+1)-2+2a}}
{J^{n(e+1)}}(\varphi_2^{-1}\otimes \varphi_1^{-1})
(w_{2n+1-f}\otimes w_{2n+2-a}), &f>0,\\
\end{array}\right.\\
&=\left\{\begin{array}{ll}
x_bx_1^{-1}(-1)^p(\bar{\mu}K)^s w_{2n+2-b}\otimes w_1, 
&a=1,\vspace{1mm}\\
\frac{x_b(-1)^p\lambda^{e+1}(\bar{\mu}K)^{s-n(e+1)-2+2a}}
{x_aJ^{n(e+1)}}w_{1+f}\otimes w_a, &f>0.\\
\end{array}\right.
\end{align*}
\noindent
Let $2\nmid a$ and $b=1$. If $a\neq 1$,   set $2a-1=2nr+d$, $r\in\{0,1\}$, $d\in\overline{1, 2n-1}$.
\begin{align*}
\bar c(w_a\otimes w_b)&=x_1 (\varphi_2^{-1}\otimes \varphi_1^{-1})c(w_{a}\otimes w_1)\\
&=\left\{\begin{array}{ll}
x_1(-1)^p(\bar{\mu}K)^s (\varphi_2^{-1}\otimes \varphi_1^{-1})(w_{1}\otimes w_1), &a=1,
\vspace{1mm}\\
x_1(-1)^p\lambda^{r+1}(\bar{\mu}K)^{s+n(r-1)}J^{n(r-1)} \left(\varphi_2^{-1}(w_{d})
   \otimes \varphi_1^{-1}(w_{2n+2-a})\right), &d>0,\\
\end{array}\right.\\
&=\left\{\begin{array}{ll}
(-1)^p(\bar{\mu}K)^s w_{1}\otimes w_1, &a=1,\\
x_1x_a^{-1}(-1)^p\lambda^{r+1}(\bar{\mu}K)^{s+n(r-1)}J^{n(r-1)} w_{d}
   \otimes w_a, &d>0.\\
\end{array}\right.
\end{align*}
\par\noindent
Let $2\nmid a$, $b\neq 1$ and $2\nmid b$. If $a\neq 1$,  set
\[
(2n+2-b)+2a-2=2n+2a-b=2nr+d, \quad r\in\{0, 1, 2\}, d\in\overline{1, 2n-1}.
\]
\begin{align*}
\bar c(w_a\otimes w_b)&=x_b (\varphi_2^{-1}\otimes \varphi_1^{-1})c(w_{a}\otimes w_{2n+2-b})\\
&=\left\{\begin{array}{ll}
x_b(-1)^p(\bar{\mu}K)^s (\varphi_2^{-1}\otimes \varphi_1^{-1})(w_{2n+2-b}\otimes w_1), &a=1,\\
x_b(-1)^p\lambda^{r+1}(\bar{\mu}K)^{s+n(r-1)}J^{n(r-1)} \left(\varphi_2^{-1}(w_{d})
   \otimes \varphi_1^{-1}(w_{2n+2-a})\right), &d>0,\\
\end{array}\right.\\
&=\left\{\begin{array}{ll}
x_bx_1^{-1}(-1)^p(\bar{\mu}K)^s w_{2n+2-b}\otimes w_1, &a=1,\\
x_bx_a^{-1}(-1)^p\lambda^{r+1}(\bar{\mu}K)^{s+n(r-1)}J^{n(r-1)} w_{d}
   \otimes w_{a}, &d>0.\\
\end{array}\right.
\end{align*}
\end{proof}

\begin{lemma}\label{NicholsEvenK_tildeC}
Let $\tilde c=(\varphi_1^{-1}\otimes \varphi_2^{-1})c(\varphi_1\otimes \varphi_2)$, and  $\Psi_E=2nr+d$, $\Psi_O=2ne+f$, where $r, e\in\Bbb N$,  $d, f\in\overline{0, 2n-1}$, and 
\begin{align*}
\Psi_E&=\left\{\begin{array}{ll}
2n+2a-b, & a\,\,\text{even},  b\,\,\text{even},\\
4n+2+b-2a, & 1\neq a\,\,\text{odd}, b\,\,\text{odd},\\
\end{array}\right.\\
\Psi_O&=\left\{\begin{array}{ll}
2n-1+2a-b, & a\,\,\text{even},  b\,\,\text{odd},\\
4n+1+b-2a, &1\neq a\,\,\text{odd},  b\,\,\text{even}.\\
\end{array}\right.
\end{align*}
\begin{enumerate}
\item If $a=1$, then 
\begin{align*}
\tilde c(w_a\otimes w_b)=\left\{\begin{array}{ll}
x_1x_{2n+2-b}^{-1}(-1)^p(\bar{\mu}K)^s w_{2n+2-b}\otimes w_1, 
   & b\,\,\text{even or } 1\neq b\,\,\text{odd},\vspace{1mm}\\
(-1)^p  (\bar{\mu}K)^s w_1\otimes w_1, & b=1.  
\end{array}\right.
\end{align*}
\item If $a\neq 1$ and $2\mid (a+b)$, then
\begin{align*}
\tilde c(w_a\otimes w_b)=\left\{\begin{array}{ll}
x_ax_{2n}^{-1}(-1)^p\lambda^r(\bar{\mu}K)^{s+n(r-2)}J^{n(r-2)} w_{2n}\otimes w_a, & d=0,\vspace{1mm}\\
x_ax_{d}^{-1}(-1)^p\lambda^{r+1}(\bar{\mu}K)^{s+n(r-1)}J^{n(r-1)} w_{d}\otimes w_a, 
    & 0\neq d\,\,\text{even},\vspace{1mm}\\
x_ax_{1}^{-1}(-1)^p\lambda^{r+1}(\bar{\mu}K)^{s+n(r-1)}J^{n(r-1)} w_{1}\otimes w_a, 
    &d=1,\vspace{1mm}\\
\dfrac{x_a(-1)^p\lambda^{r+1}(\bar{\mu}K)^{s+n(r-1)}J^{n(r-1)}}
          {x_{2n+2-d}}    w_{2n+2-d}\otimes w_a, 
          &1\neq d\,\,\text{odd}.
\end{array}\right.
\end{align*}
\item If $a\neq 1$ and $2\nmid (a+b)$, then
\begin{align*}
\tilde c(w_a\otimes w_b)=\left\{\begin{array}{ll}
x_ax_{1}^{-1}(-1)^p\lambda^e(\bar{\mu}K)^{s-ne+2a-2}J^{-ne} w_{1}\otimes w_a, & f=0,\vspace{1mm}\\
\dfrac{x_a(-1)^p\lambda^{e+1}(\bar{\mu}K)^{s-n(e+1)+2a-2}}
          {x_{f+1}J^{n(e+1)}} w_{f+1}\otimes w_a, 
    & 0\neq f\,\,\text{even},\vspace{1mm}\\
\dfrac{x_a(-1)^p\lambda^{e+1}(\bar{\mu}K)^{s+n(3-e)-2a+2}}
          {x_{2n+1-f}J^{n(e+1)}}    w_{2n+1-f}\otimes w_a, 
          &f\,\,\text{odd}.
\end{array}\right.
\end{align*}
\end{enumerate}
\end{lemma}

\begin{proof}
It is a case by case verification similar to the proof of Lemma \ref{NicholsEvenK_tildeC}.
\end{proof}

\subsection{The Nichols algebras $\mathfrak{B}\left(\mathscr{N}_{k,pq}^s\right)$}
Let $s\in\overline{1,N}$, 
    $k\in\overline{0,N-1}$, $p,q\in\Bbb{Z}_2$, and $\mathscr{N}_{k,pq}^s$ be the
$(2n+1)$-dimensional 
simple Yetter-Drinfeld module 
over $A_{N\, 2n+1}^{\mu\lambda}$ as defined in \cite[Appendix]{Shi2020odd}.
 Set $B=\bar{\mu}^{\frac12}\omega^{2k(2n+1)}$, where the parameter $\omega$ is a $8N(2n+1)$-th primitive 
 root of unity. 
 We define two maps $R^{\gamma }, L^{\gamma }: \mathscr{N}_{k,pq}^s\otimes \mathscr{N}_{k,pq}^s
\longrightarrow \mathscr{N}_{k,pq}^s\otimes \mathscr{N}_{k,pq}^s$ such that 
\begin{align*}
R^{\gamma }(w_\alpha\otimes w_\beta)
&=\left\{\begin{array}{rl}
w_{\beta+\gamma }\otimes w_{2n-\alpha+2},
&\beta+\gamma \leq 2n+1,\\
(-1)^p\lambda B
w_{2n+1}\otimes w_{2n-\alpha+2},
&\beta+\gamma =2n+2,\\
(-1)^p\lambda B^{2(\gamma +\beta)-4n-3}
w_{4n+3-\gamma -\beta}\otimes w_{2n-\alpha+2},
&\beta+\gamma\geq 2n+3, \\
\end{array}\right.
\\
L^{\gamma }(w_\alpha\otimes w_\beta)
&=\left\{\begin{array}{rl}
B^{2\gamma}
w_{\beta-\gamma}\otimes w_{2n-\alpha+2}, 
&\gamma<\beta,\vspace{1mm}\\
(-1)^p B^{2\beta-1}
w_{\gamma-\beta+1}\otimes w_{2n-\alpha+2}, 
&\gamma\in\overline{\beta,\beta+2n}.\\
\end{array}\right.
\end{align*}
Then the braiding of $\mathfrak{B}\left(\mathscr{N}_{k,pq}^s\right)$ can 
be described as follows \cite[Section 4.4]{Shi2020odd}. 
\begin{enumerate}
\item If $\alpha=n+1$, then 
\[
c(w_\alpha\otimes w_\beta)
=(-1)^q B^{2(\alpha+s-1)}
w_\beta\otimes w_{2n-\alpha+2};
\]
\item If $\alpha<n+1$, 
          $\alpha+\beta \equiv 0\mod 2$,
         then 
\[
c(w_\alpha\otimes w_\beta)
=(-1)^q B^{2(2\alpha+s-n-2)}
L^{2(n-\alpha+1)}(w_\alpha\otimes w_\beta);
\]
\item If $\alpha>n+1$, 
          $\alpha+\beta\equiv 1\mod 2$,
         then 
\[
c(w_\alpha\otimes w_\beta)
=(-1)^q B^{2(s+n)}
L^{2(\alpha-1-n)}(w_\alpha\otimes w_\beta);
\]

\item If $\alpha<n+1$, 
          $\alpha+\beta \equiv 1\mod 2$,
         then 
\[
c(w_\alpha\otimes w_\beta)
=(-1)^q B^{2(2\alpha+s-n-2)}
R^{2(n-\alpha+1)}(w_\alpha\otimes w_\beta);
\]
\item If $\alpha>n+1$, 
          $\alpha+\beta\equiv 0\mod 2$,
         then 
\[
c(w_\alpha\otimes w_\beta)
=(-1)^q B^{2(s+n)}
R^{2(\alpha-1-n)}(w_\alpha\otimes w_\beta).
\]
\end{enumerate}

\begin{theorem}\label{NicholsAlgN}
The braided vector space $\left(\mathscr{N}_{k,pq}^s, c\right)$ is t-equivalent to a braided vector space 
associated to dihedral rack $D_{2n+1}$. 
\end{theorem}
\begin{proof}
Set   $x_{2k}=\lambda x_1 B^{4n-4k+2}$ and 
$x_{2k+1}=x_1 B^{4k}$ for  $k\in\overline{1, n}$, 
then $\tilde c=\bar c$. 
It is a case by case verification to prove $\tilde c=\bar c$, see Lemmas \ref{NicholsOddN_tildeC}
and \ref{NicholsOddN_BarC}.  
For example, if $a<n+1$, $2\mid a$, $2\nmid b$, then 
\begin{align*}
\tilde c(w_a\otimes w_b)&=\left\{\begin{array}{ll}
\dfrac{x_a(-1)^qB^{4a+2s-2n-4}}{x_{2a-b}} w_{2a-b}\otimes w_a, & 2a-b>0,\vspace{1mm}\\
\dfrac{x_a(-1)^{p+q}\lambda B^{-2n-3+2b+2s}}{x_{2n+1+2a-b}} w_{2n+1+2a-b}\otimes w_a, & 2a-b\leq 0,\\
\end{array}\right. 
\end{align*}
\begin{align*}
\bar c(w_a\otimes w_b)&=\left\{\begin{array}{ll}
\dfrac{x_b(-1)^qB^{6n+4-4a+2s}}{x_{a}} w_{2a-b}\otimes w_a, & 2a-b>0,\vspace{1mm}\\
\dfrac{x_b(-1)^{p+q} B^{6n+3-2b+2s}}{x_{a}} w_{2n+1+2a-b}\otimes w_a, & 2a-b< 0.\\
\end{array}\right. 
\end{align*}
If $2a-b>0$, then we have
\[
\dfrac{x_a^2}{x_{2a-b}x_b}
=\dfrac{\left[\lambda x_1B^{4n-2a+2}\right]^2}{x_1B^{2(2a-b-1)}\cdot x_1B^{2(b-1)}}
=B^{8n+8-8a}.
\]
If $2a-b<0$, then we have 
\[
\dfrac{x_a^2}{x_{2n+1+2a-b}x_b}
=\dfrac{\left[\lambda x_1B^{4n-2a+2}\right]^2}
{\lambda x_1B^{4n-2(2n+1+2a-b)+2}\cdot x_1B^{2(b-1)}}
=B^{8n+6-4b}.
\]
It is easy to see that $\tilde c(w_a\otimes w_b)=\bar c(w_a\otimes w_b)$ in case 
$a<n+1$, $2\mid a$ and $2\nmid b$.

According to Lemma \ref{MainLemma}, 
$\left(\mathscr{N}_{k,pq}^s, c\right)$ and 
$\left(\mathscr{N}_{k,pq}^s, \tilde c\right)$ are t-equivalent. From  Lemma \ref{NicholsOddN_tildeC}, it is easy to see that the braided vector space 
$\left(\mathscr{N}_{k,pq}^s, \tilde c\right)$ is associated to dihedral rack $D_{2n+1}$. 
\end{proof}

\par\noindent
Define two invertible maps $\varphi_i: \mathscr{N}_{k,pq}^s\rightarrow \mathscr{N}_{k,pq}^s$ 
for $i=1, 2$, via 
\begin{align*}
\varphi_1(w_a)&=\left\{\begin{array}{ll}
x_aw_{2n+2-a}, & a \,\,\text{odd, }x_a\in\k^{\times}, \\
x_a w_a, & a \,\,\text{even, } x_a\in\k^{\times}, \\
\end{array}\right. \\
\varphi_2(w_a)&=\left\{\begin{array}{ll}
w_{a}, & a \,\,\text{odd}, \\
w_{2n+2-a}, & a \,\,\text{even}.\\
\end{array}\right. 
\end{align*}

\begin{lemma}\label{NicholsOddN_tildeC}
Let $\tilde c=(\varphi_1^{-1}\otimes \varphi_2^{-1})c(\varphi_1\otimes \varphi_2)$. 
\begin{enumerate}
\item If $a=n+1$, then 
\begin{align*}
\tilde c(w_a\otimes w_b)=(-1)^qx_{n+1}x_{2n+2-b}^{-1}B^{2n+2s} w_{2n+2-b}\otimes w_{n+1}.
\end{align*}
\item If $a<n+1$, $2a-b>0$, then 
\begin{align*}
\tilde c(w_a\otimes w_b)&=\left\{\begin{array}{ll}
\dfrac{x_a(-1)^qB^{2n+2s}}{x_{2a-b}}w_{2a-b}\otimes w_a, & 2\mid (a+b),\vspace{1mm}\\
\dfrac{x_a(-1)^qB^{4a+2s-2n-4}}{x_{2a-b}}w_{2a-b}\otimes w_a, & 2\mid a, 2\nmid b,\vspace{1mm}\\
\dfrac{x_a(-1)^qB^{6n+4+2s-4a}}{x_{2a-b}}w_{2a-b}\otimes w_a, & 2\nmid a, 2\mid b.\\
\end{array}\right.
\end{align*}

\item If $a<n+1$, $2a-b\leq 0$, then 
\begin{align*}
\tilde c(w_a\otimes w_b)&=\left\{\begin{array}{ll}
\dfrac{x_a(-1)^{p+q}B^{2n-1+4a-2b+2s}}{x_{2n+1+2a-b}} 
          w_{2n+1+2a-b}\otimes w_a, & 2\mid a, 2\mid b,\vspace{1mm}\\
\dfrac{x_a(-1)^{p+q}\lambda B^{2n+1-4a+2b+2s}}{x_{2n+1+2a-b}} 
          w_{2n+1+2a-b}\otimes w_a, & 2\nmid a, 2\nmid b,\vspace{1mm}\\
\dfrac{x_a(-1)^{p+q}\lambda B^{-2n-3+2b+2s}}{x_{2n+1+2a-b}} 
          w_{2n+1+2a-b}\otimes w_a, & 2\mid a, 2\nmid b,\vspace{1mm}\\     
\dfrac{x_a(-1)^{p+q} B^{6n+3-2b+2s}}{x_{2n+1+2a-b}} 
          w_{2n+1+2a-b}\otimes w_a, & 2\nmid a, 2\mid b.\\               
\end{array}\right.
\end{align*}

\item If $a>n+1$, $2a-b\leq 2n+1$, then 
\begin{align*}
\tilde c(w_a\otimes w_b)&=\left\{\begin{array}{ll}
x_ax_{2a-b}^{-1}(-1)^qB^{2n+2s}w_{2a-b}\otimes w_a, &2\mid (a+b), \vspace{1mm}\\
\dfrac{x_a(-1)^qB^{-2n-4+4a+2s}}{x_{2a-b}}w_{2a-b}\otimes w_a, 
             & 2\mid a, 2\nmid b,\vspace{1mm}\\
\dfrac{x_a(-1)^qB^{6n+4-4a+2s}}{x_{2a-b}}w_{2a-b}\otimes w_a, 
             & 2\nmid a, 2\mid b.\\
\end{array}\right.
\end{align*}

\item If $a>n+1$, $2a-b= 2n+2$, then 
\begin{align*}
\tilde c(w_a\otimes w_b)&=\left\{\begin{array}{ll}
\dfrac{x_a(-1)^{p+q}\lambda B^{2n+1+2s}}{x_1} w_1\otimes w_a, & 2\mid a, 2\mid b,\vspace{1mm}\\
\dfrac{x_a(-1)^{p+q}\lambda B^{6n+5-4a+2s}}{x_1} w_1\otimes w_a, & 2\nmid a, 2\mid b.\\
\end{array}\right.
\end{align*}

\item If $a>n+1$, $2a-b >2n+2$, then 
\begin{align*}
\tilde c(w_a\otimes w_b)&=\left\{\begin{array}{ll}
\dfrac{x_a(-1)^{p+q}\lambda B^{-2n-3+4a-2b+2s}}{x_{2a-b-2n-1}}
        w_{2a-b-2n-1}\otimes w_a, & 2\mid a, 2\mid b, \vspace{1mm}\\
\dfrac{x_a(-1)^{p+q} B^{6n+3-4a+2b+2s}}{x_{2a-b-2n-1}}
        w_{2a-b-2n-1}\otimes w_a, & 2\nmid a, 2\nmid b, \vspace{1mm}\\ 
\dfrac{x_a(-1)^{p+q} B^{2n-1+2b+2s}}{x_{2a-b-2n-1}}
        w_{2a-b-2n-1}\otimes w_a, & 2\mid a, 2\nmid b, \vspace{1mm}\\    
\dfrac{x_a(-1)^{p+q}\lambda B^{2n+1-2b+2s}}{x_{2a-b-2n-1}}
        w_{2a-b-2n-1}\otimes w_a, & 2\nmid a, 2\mid b.\\                   
\end{array}\right.
\end{align*}
\end{enumerate}
\end{lemma}

\begin{proof}
If $a=n+1$, then 
\begin{align*}
\tilde c(w_{n+1}\otimes w_b)
&=\left\{\begin{array}{ll}
(\varphi_1^{-1}\otimes \varphi_2^{-1})c(x_{n+1}w_{n+1}\otimes w_b), & b\,\,\text{odd},\\
(\varphi_1^{-1}\otimes \varphi_2^{-1})c(x_{n+1}w_{n+1}\otimes w_{2n+2-b}), & b\,\,\text{even},\\
\end{array}\right.\\
&=\left\{\begin{array}{ll}
x_{n+1} (-1)^qB^{2(n+1+s-1)}(\varphi_1^{-1}\otimes \varphi_2^{-1})(w_{b}\otimes w_{n+1}), & b\,\,\text{odd},\\
x_{n+1} (-1)^qB^{2(n+1+s-1)}(\varphi_1^{-1}\otimes \varphi_2^{-1})(w_{2n+2-b}\otimes w_{n+1}), & b\,\,\text{even},\\
\end{array}\right.\\
&=x_{n+1}x_{2n+2-b}^{-1} (-1)^qB^{2(n+s)}w_{2n+2-b}\otimes w_{n+1}.
\end{align*}
\noindent
If $a<n+1$, $2\mid a$ and $2\mid b$, then 
\begin{align*}
&\quad \tilde c(w_a\otimes w_b)=(\varphi_1^{-1}\otimes \varphi_2^{-1})c(x_aw_{a}\otimes w_{2n+2-b})\\
&=(-1)^qx_aB^{2(2a+s-n-2)}(\varphi_1^{-1}\otimes \varphi_2^{-1})L^{2(n-a+1)}(w_a\otimes w_{2n+2-b})\\
&=\left\{\begin{array}{ll}
(-1)^qx_aB^{2n+2s}(\varphi_1^{-1}\otimes \varphi_2^{-1})
    (w_{2a-b}\otimes w_{2n+2-a}), &2a-b>0,\vspace{1mm}\\
(-1)^{p+q}x_aB^{2n-1+4a-2b+2s}(\varphi_1^{-1}\otimes \varphi_2^{-1})
    (w_{b-2a+1}\otimes w_{2n+2-a}), &2a-b\leq 0,
\end{array}\right.\\
&=\left\{\begin{array}{ll}
(-1)^qx_ax_{2a-b}^{-1}B^{2n+2s}
    w_{2a-b}\otimes w_{a}, &2a-b>0,\vspace{1mm}\\
(-1)^{p+q}x_ax_{2n+1+2a-b}^{-1}B^{2n-1+4a-2b+2s}
    w_{2n+1+2a-b}\otimes w_{a}, &2a-b\leq 0.
\end{array}\right.
\end{align*}

\par\noindent
If $a< n+1$, $2\nmid a$, $2\nmid b$, then 
\begin{align*}
&\quad \tilde c(w_a\otimes w_b)=(\varphi_1^{-1}\otimes \varphi_2^{-1})c(x_aw_{2n+2-a}\otimes w_{b})\\
&=(-1)^qx_aB^{2s+2n}(\varphi_1^{-1}\otimes \varphi_2^{-1})
R^{2n+2-2a}(w_{2n+2-a}\otimes w_{b})\\
&=\left\{\begin{array}{ll}
(-1)^qx_aB^{2s+2n}(\varphi_1^{-1}\otimes \varphi_2^{-1})(w_{2n+2+b-2a}\otimes w_{a}), 
       &2a-b\geq 1,\vspace{1mm}\\
(-1)^{p+q}x_a\lambda B^{2n+1+2s+2b-4a}(\varphi_1^{-1}\otimes \varphi_2^{-1})
       (w_{2n+1+2a-b}\otimes w_{a}), &2a-b<0,\\
\end{array}\right.\\
&=\left\{\begin{array}{ll}
(-1)^qx_ax_{2a-b}^{-1}B^{2s+2n}w_{2a-b}\otimes w_{a}, 
       &2a-b>0,\vspace{1mm}\\
(-1)^{p+q}x_ax_{2n+1+2a-b}^{-1}\lambda B^{2n+1+2s+2b-4a}
       w_{2n+1+2a-b}\otimes w_{a}, &2a-b<0.\\
\end{array}\right.
\end{align*}

\par\noindent 
If $a<n+1$, $2\mid a$, $2\nmid b$, then 
\begin{align*}
&\quad \tilde c(w_a\otimes w_b)=(\varphi_1^{-1}\otimes \varphi_2^{-1})c(x_aw_{a}\otimes w_{b})\\
&=(-1)^qx_aB^{4a+2s-2n-4}(\varphi_1^{-1}\otimes \varphi_2^{-1})
R^{2n+2-2a}(w_{a}\otimes w_{b})\\
&=\left\{\begin{array}{ll}
(-1)^qx_aB^{4a+2s-2n-4}(\varphi_1^{-1}\otimes \varphi_2^{-1})(w_{2n+2-2a+b}\otimes w_{2n+2-a}), 
      &2a-b>0,\vspace{1mm}\\
(-1)^{p+q}x_a\lambda B^{2s+2b-2n-3}(\varphi_1^{-1}\otimes \varphi_2^{-1})
       (w_{2n+1+2a-b}\otimes w_{2n+2-a}), 
      &2a-b<0,
\end{array}\right.\\
&=\left\{\begin{array}{ll}
(-1)^qx_ax_{2a-b}^{-1}B^{4a+2s-2n-4}w_{2a-b}\otimes w_{a}, 
      &2a-b>0,\vspace{1mm}\\
(-1)^{p+q}x_ax_{2n+1+2a-b}^{-1}\lambda B^{2s+2b-2n-3}w_{2n+1+2a-b}\otimes w_{a},
      &2a-b<0.
\end{array}\right.
\end{align*}

\par\noindent
If $a<n+1$, $2\nmid a$, $2\mid b$, then 
\begin{align*}
&\quad \tilde c(w_a\otimes w_b)=(\varphi_1^{-1}\otimes \varphi_2^{-1})c(x_aw_{2n+2-a}\otimes w_{2n+2-b})\\
&=(-1)^qx_aB^{2n+2s}(\varphi_1^{-1}\otimes \varphi_2^{-1})
L^{2n+2-2a}(w_{2n+2-a}\otimes w_{2n+2-b})\\
&=\left\{\begin{array}{ll}
(-1)^qx_aB^{2n+2s+2(2n+2-2a)}(\varphi_1^{-1}\otimes \varphi_2^{-1})
       (w_{2a-b}\otimes w_{a}), &2a-b>0,\vspace{1mm}\\
(-1)^{p+q}x_aB^{2n+2s+2(2n+2-b)-1}(\varphi_1^{-1}\otimes \varphi_2^{-1})
        (w_{b-2a+1}\otimes w_{a}) &2a-b\leq 0,\\
\end{array}\right.\\
&=\left\{\begin{array}{ll}
(-1)^qx_ax_{2a-b}^{-1}B^{6n+4+2s-4a}
        w_{2a-b}\otimes w_{a}, &2a-b>0,\vspace{1mm}\\
(-1)^{p+q}x_ax_{2n+1+2a-b}^{-1}B^{6n+3+2s-2b}
        w_{2n+1+2a-b}\otimes w_{a},  &2a-b\leq 0.\\
\end{array}\right.
\end{align*}

\par\noindent 
If $a>n+1$, $2\mid a$, $2\mid b$, then 
\begin{align*}
&\quad \tilde c(w_a\otimes w_b)=(\varphi_1^{-1}\otimes \varphi_2^{-1})c(x_aw_{a}\otimes w_{2n+2-b})\\
&=(-1)^qx_aB^{2s+2n}(\varphi_1^{-1}\otimes \varphi_2^{-1})R^{2a-2n-2}(w_a\otimes w_{2n+2-b})\\
&=\left\{\begin{array}{ll}
(-1)^qx_aB^{2s+2n}(\varphi_1^{-1}\otimes \varphi_2^{-1})(w_{2a-b}\otimes w_{2n+2-a}), 
         & 2a-b\leq 2n+1,\vspace{1mm}\\
(-1)^{p+q}x_a\lambda B^{2s+2n+1}(\varphi_1^{-1}\otimes \varphi_2^{-1})
        (w_{2n+1}\otimes w_{2n+2-a}), & 2a-b= 2n+2,\vspace{1mm}\\
(-1)^{p+q}x_a\lambda B^{4a-2b+2s-2n-3}(\varphi_1^{-1}\otimes \varphi_2^{-1})
        (w_{4n+3-2a+b}\otimes w_{2n+2-a}),& 2a-b\geq 2n+3,\\
\end{array}\right. \\
&=\left\{\begin{array}{ll}
\dfrac{(-1)^qx_aB^{2s+2n}}{x_{2a-b}}w_{2a-b}\otimes w_{a}, 
         & 2a-b\leq 2n+1,\vspace{1mm}\\
\dfrac{(-1)^{p+q}x_a\lambda B^{2s+2n+1}}{x_1}
        w_{1}\otimes w_{a}, & 2a-b= 2n+2,\vspace{1mm}\\
\dfrac{(-1)^{p+q}x_a\lambda B^{4a-2b+2s-2n-3}}{x_{2a-b-2n-1}}
        w_{2a-b-2n-1}\otimes w_{a}, & 2a-b\geq 2n+3.\\
\end{array}\right.
\end{align*}

\par \noindent 
If $a>n+1$, $2\nmid a$, $2\nmid b$, then 
\begin{align*}
&\quad \tilde c(w_a\otimes w_b)=(\varphi_1^{-1}\otimes \varphi_2^{-1})c(x_aw_{2n+2-a}\otimes w_{b})\\
&=(-1)^qx_aB^{2(3n+2-2a+s)}(\varphi_1^{-1}\otimes \varphi_2^{-1})L^{2a-2n-2}(w_{2n+2-a}\otimes w_{b})\\
&=\left\{\begin{array}{ll}
(-1)^qx_aB^{2n+2s}(\varphi_1^{-1}\otimes \varphi_2^{-1})
         (w_{2n+2-2a+b}\otimes w_a), &2a-b<2n+2,\vspace{1mm}\\
(-1)^{p+q}x_aB^{2(3n+2-2a+s)+2b-1}(\varphi_1^{-1}\otimes \varphi_2^{-1})
         (w_{2a-b-2n-1}\otimes w_a), &2a-b\geq 2n+2,\\
\end{array}\right.\\
&=\left\{\begin{array}{ll}
\dfrac{(-1)^qx_aB^{2n+2s}}{x_{2a-b}}
         w_{2a-b}\otimes w_a, &2a-b<2n+2,\vspace{1mm}\\
\dfrac{(-1)^{p+q}x_aB^{6n+3-4a+2b+2s}}{x_{2a-b-2n-1}}
         w_{2a-b-2n-1}\otimes w_a, &2a-b\geq 2n+2.\\
\end{array}\right.
\end{align*}
\par\noindent
If $a>n+1$, $2\mid a$, $2\nmid b$, then 
\begin{align*}
&\quad \tilde c(w_a\otimes w_b)=(\varphi_1^{-1}\otimes \varphi_2^{-1})c(x_aw_{a}\otimes w_{b})\\
&=(-1)^qx_aB^{2n+2s}(\varphi_1^{-1}\otimes \varphi_2^{-1})L^{2a-2n-2}(w_{a}\otimes w_{b})\\
&=\left\{\begin{array}{ll}
(-1)^qx_aB^{2(2a-n-2+s)}(\varphi_1^{-1}\otimes \varphi_2^{-1})
        (w_{b-2a+2n+2}\otimes w_{2n+2-a}), &2a-b<2n+2,\vspace{1mm}\\
(-1)^{p+q}x_aB^{2n+2s+2b-1}(\varphi_1^{-1}\otimes \varphi_2^{-1})
        (w_{2a-b-2n-1}\otimes w_{2n+2-a}), &2a-b\geq 2n+2,\\
\end{array}\right.\\
&=\left\{\begin{array}{ll}
\dfrac{(-1)^qx_aB^{2(2a-n-2+s)}}{x_{2a-b}}
        w_{2a-b}\otimes w_{a}, &2a-b<2n+2,\vspace{1mm}\\
\dfrac{(-1)^{p+q}x_aB^{2n+2s+2b-1}}{x_{2a-b-2n-1}}
        w_{2a-b-2n-1}\otimes w_{a}, &2a-b\geq 2n+2.\\
\end{array}\right.
\end{align*}

\par\noindent
If $a>n+1$, $2\nmid a$, $2\mid b$, then 
\begin{align*}
&\quad \tilde c(w_a\otimes w_b)=(\varphi_1^{-1}\otimes \varphi_2^{-1})c(x_aw_{2n+2-a}\otimes w_{2n+2-b})\\
&=(-1)^qx_aB^{2(3n+2-2a+s)}(\varphi_1^{-1}\otimes \varphi_2^{-1})
      R^{2a-2n-2}(w_{2n+2-a}\otimes w_{2n+2-b})\\
&=\left\{\begin{array}{ll}
(-1)^qx_aB^{2(3n+2-2a+s)}(\varphi_1^{-1}\otimes \varphi_2^{-1})
       (w_{2a-b}\otimes w_a), &2a-b\leq 2n+1,\vspace{1mm}\\
(-1)^{p+q}x_a\lambda B^{2(3n+2-2a+s)+1}(\varphi_1^{-1}\otimes \varphi_2^{-1})
       (w_{2n+1}\otimes w_a), &2a-b=2n+2,\vspace{1mm}\\
(-1)^{p+q}x_a\lambda B^{2n+1-2b+2s}(\varphi_1^{-1}\otimes \varphi_2^{-1})
       (w_{4n+3-2a+b}\otimes w_a), &2a-b\geq 2n+3,
\end{array}\right.  \\
&=\left\{\begin{array}{ll}
(-1)^qx_aB^{2(3n+2-2a+s)}x_{2a-b}^{-1}
       w_{2a-b}\otimes w_a, &2a-b\leq 2n+1,\vspace{1mm}\\
(-1)^{p+q}x_a\lambda B^{2(3n+2-2a+s)+1}x_1^{-1}
       w_{1}\otimes w_a, &2a-b=2n+2,\vspace{1mm}\\
\dfrac{(-1)^{p+q}x_a\lambda B^{2n+1-2b+2s}}{x_{2a-b-2n-1}}
       w_{2a-b-2n-1}\otimes w_a, &2a-b\geq 2n+3.
\end{array}\right.     
\end{align*}
\end{proof}

\begin{lemma}\label{NicholsOddN_BarC}
Let $\bar c=(\varphi_2^{-1}\otimes \varphi_1^{-1})c(\varphi_2\otimes \varphi_1)$.  
\begin{enumerate}
\item If $a=n+1$, then 
\[
\bar c(w_a\otimes w_b)=x_{b}x_{n+1}^{-1}(-1)^qB^{2n+2s}w_{2n+2-b}\otimes w_{n+1}.
\]
\item If $a<n+1$ and $2a-b>0$, then 
\begin{align*}
\bar c(w_a\otimes w_b)=\left\{\begin{array}{ll}
x_bx_a^{-1}(-1)^qB^{2n+2s} w_{2a-b}\otimes w_a, & 2\mid (a+b), \vspace{1mm}\\
x_bx_a^{-1}(-1)^qB^{-2n-4+4a+2s} w_{2a-b}\otimes w_a, & 2\nmid a, 2\mid b, \vspace{1mm}\\
x_bx_a^{-1}(-1)^qB^{6n+4-4a+2s} w_{2a-b}\otimes w_a, & 2\mid a, 2\nmid b.\vspace{1mm}\\
\end{array}\right. 
\end{align*}
\item If $a<n+1$ and $2a-b=0$, then 
\begin{align*}
\bar c(w_a\otimes w_b)=\left\{\begin{array}{ll}
x_bx_a^{-1}(-1)^{p+q}\lambda B^{2n+1+2s} w_{2n+1}\otimes w_a, & 2\mid a, 2\mid b, \vspace{1mm}\\
x_bx_a^{-1}(-1)^{p+q}\lambda B^{-2n-3+4a+2s} w_{2n+1}\otimes w_a, & 2\nmid a, 2\mid b. \vspace{1mm}\\
\end{array}\right. 
\end{align*}

\item If $a<n+1$ and $2a-b<0$, then 
\begin{align*}
\bar c(w_a\otimes w_b)=\left\{\begin{array}{ll}
x_bx_a^{-1}(-1)^{p+q}\lambda  B^{2n+1-4a+2b+2s} 
                  w_{2n+1+2a-b}\otimes w_a, & 2\mid a, 2\mid b, \vspace{1mm}\\
x_bx_a^{-1}(-1)^{p+q}  B^{2n-1+4a-2b+2s} 
                  w_{2n+1+2a-b}\otimes w_a, & 2\nmid a, 2\nmid b, \vspace{1mm}\\  
x_bx_a^{-1}(-1)^{p+q}\lambda  B^{-2n-3+2b+2s} 
                  w_{2n+1+2a-b}\otimes w_a, & 2\nmid a, 2\mid b, \vspace{1mm}\\   
x_bx_a^{-1}(-1)^{p+q}  B^{6n+3-2b+2s} 
                  w_{2n+1+2a-b}\otimes w_a, & 2\mid a, 2\nmid b. \vspace{1mm}\\                                                 
\end{array}\right.
\end{align*}

\item If $a>n+1$ and $2a-b<2n+2$, then 
\begin{align*}
\bar c(w_a\otimes w_b)=\left\{\begin{array}{ll}
x_bx_a^{-1}(-1)^{q} B^{2n+2s} w_{2a-b}\otimes w_a, & 2\mid (a+b), \vspace{1mm}\\
x_bx_a^{-1}(-1)^{q} B^{6n+4-4a+2s} w_{2a-b}\otimes w_a, & 2\mid a, 2\nmid b, \vspace{1mm}\\
x_bx_a^{-1}(-1)^{q} B^{-2n-4+4a+2s} w_{2a-b}\otimes w_a, & 2\nmid a, 2\mid b. \vspace{1mm}\\
\end{array}\right. 
\end{align*}
\item If $a>n+1$ and $2a-b\geq 2n+2$, then 
\begin{align*}
\bar c(w_a\otimes w_b)=\left\{\begin{array}{ll}
x_bx_a^{-1}(-1)^{p+q} B^{6n+3-4a+2b+2s} 
        w_{2a-b-2n-1}\otimes w_a, & 2\mid a, 2\mid b, \vspace{1mm}\\
x_bx_a^{-1}(-1)^{p+q} \lambda B^{-2n-3+4a-2b+2s} 
        w_{2a-b-2n-1}\otimes w_a, & 2\nmid a, 2\nmid b, \vspace{1mm}\\
x_bx_a^{-1}(-1)^{p+q}\lambda B^{2n+1-2b+2s} 
        w_{2a-b-2n-1}\otimes w_a, & 2\mid a, 2\nmid b, \vspace{1mm}\\
x_bx_a^{-1}(-1)^{p+q} B^{2n-1+2b+2s} w_{2a-b-2n-1}\otimes w_a, & 2\nmid a, 2\mid b. \vspace{1mm}\\
\end{array}\right. 
\end{align*}
\end{enumerate}
\end{lemma}

\begin{proof}
If $a=n+1$, then 
\begin{align*}
\bar c(w_a\otimes w_b)&=\left\{\begin{array}{ll}
(\varphi_2^{-1}\otimes \varphi_1^{-1})c(x_{b}w_{n+1}\otimes w_b), &2\mid b,\\
(\varphi_2^{-1}\otimes \varphi_1^{-1})c(x_{b}w_{n+1}\otimes w_{2n+2-b}), &2\nmid b,\\
\end{array}\right.\\
&=\left\{\begin{array}{ll}
x_{b}(-1)^qB^{2n+2s}(\varphi_2^{-1}\otimes \varphi_1^{-1})(w_{b}\otimes w_{n+1}), &2\mid b,\\
x_{b}(-1)^qB^{2n+2s}(\varphi_2^{-1}\otimes \varphi_1^{-1})
        (w_{2n+2-b}\otimes w_{n+1}), &2\nmid b,\\
\end{array}\right.\\
&=x_{b}x_{n+1}^{-1}(-1)^qB^{2n+2s}w_{2n+2-b}\otimes w_{n+1}.
\end{align*}
\par\noindent
If $a<n+1$, $2\mid a$, $2\mid b$, then 
\begin{align*}
&\quad \bar c(w_a\otimes w_b)=(\varphi_2^{-1}\otimes \varphi_1^{-1})c(x_{b}w_{2n+2-a}\otimes w_b)\\
&=x_{b}(-1)^qB^{2n+2s}(\varphi_2^{-1}\otimes \varphi_1^{-1})R^{2n+2-2a}(w_{2n+2-a}\otimes w_b)\\
&=\left\{\begin{array}{ll}
x_{b}(-1)^qB^{2n+2s}(\varphi_2^{-1}\otimes \varphi_1^{-1})
          (w_{2n+2-2a+b}\otimes w_a),&2a-b>0,\vspace{1mm}\\
x_{b}(-1)^{p+q}\lambda B^{2n+2s+1}(\varphi_2^{-1}\otimes \varphi_1^{-1})
          (w_{2n+1}\otimes w_a),&2a-b=0,\vspace{1mm}\\
x_{b}(-1)^{p+q}\lambda B^{2n+1-4a+2b+2s}(\varphi_2^{-1}\otimes \varphi_1^{-1})
          (w_{2n+1+2a-b}\otimes w_a),&2a-b<0,\\
\end{array}\right.\\
&=\left\{\begin{array}{ll}
x_{b}x_a^{-1}(-1)^qB^{2n+2s}
          w_{2a-b}\otimes w_a,&2a-b>0,\vspace{1mm}\\
x_{b}x_a^{-1}(-1)^{p+q}\lambda B^{2n+2s+1}
          w_{2n+1}\otimes w_a,&2a-b=0,\vspace{1mm}\\
x_{b}x_a^{-1}(-1)^{p+q}\lambda B^{2n+1-4a+2b+2s}
          w_{2n+1+2a-b}\otimes w_a,&2a-b<0.\\
\end{array}\right.
\end{align*}

\par\noindent
If $a<n+1$, $2\nmid a$, $2\nmid b$, then 
\begin{align*}
&\quad \bar c(w_a\otimes w_b)=(\varphi_2^{-1}\otimes \varphi_1^{-1})c(x_{b}w_{a}\otimes w_{2n+2-b})\\
&=x_{b}(-1)^qB^{-2n-4+4a+2s}(\varphi_2^{-1}\otimes \varphi_1^{-1})
     L^{2n+2-2a}(w_{a}\otimes w_{2n+2-b})\\
&=\left\{\begin{array}{ll}
x_{b}(-1)^qB^{2n+2s}(\varphi_2^{-1}\otimes \varphi_1^{-1})
      (w_{2a-b}\otimes w_{2n+2-a}), &2a-b>0,\vspace{1mm}\\
x_{b}(-1)^{p+q}B^{2n-1+4a-2b+2s}(\varphi_2^{-1}\otimes \varphi_1^{-1})
      (w_{-2a+b+1}\otimes w_{2n+2-a}), &2a-b<0,\\
\end{array}\right.\\
&=\left\{\begin{array}{ll}
x_{b}x_a^{-1}(-1)^qB^{2n+2s}
      w_{2a-b}\otimes w_{a}, &2a-b>0,\vspace{1mm}\\
x_{b}x_a^{-1}(-1)^{p+q}B^{2n-1+4a-2b+2s}
      w_{2n+1+2a-b}\otimes w_{a}, &2a-b<0.\\
\end{array}\right.
\end{align*}

\par\noindent
If $a<n+1$, $2\nmid a$, $2\mid b$, then 
\begin{align*}
&\quad \bar c(w_a\otimes w_b)=(\varphi_2^{-1}\otimes \varphi_1^{-1})c(x_{b}w_{a}\otimes w_{b})\\
&=x_{b}(-1)^qB^{-2n-4+4a+2s}(\varphi_2^{-1}\otimes \varphi_1^{-1})
     R^{2n+2-2a}(w_{a}\otimes w_{b})\\
&=\left\{\begin{array}{ll}
x_{b}(-1)^qB^{-2n-4+4a+2s}(\varphi_2^{-1}\otimes \varphi_1^{-1})
     (w_{2n+2-2a+b}\otimes w_{2n+2-a}), & 2a-b>0, \vspace{1mm}\\
x_{b}(-1)^{p+q}\lambda B^{-2n-4+4a+2s+1}(\varphi_2^{-1}\otimes \varphi_1^{-1})
     (w_{2n+1}\otimes w_{2n+2-a}), & 2a-b=0, \vspace{1mm}\\
x_{b}(-1)^{p+q}\lambda B^{-2n-3+2b+2s}(\varphi_2^{-1}\otimes \varphi_1^{-1})
     (w_{2n+1+2a-b}\otimes w_{2n+2-a}), & 2a-b<0, \vspace{1mm}\\
\end{array}\right.  \\
&=\left\{\begin{array}{ll}
x_{b}x_a^{-1}(-1)^qB^{-2n-4+4a+2s}
     w_{2a-b}\otimes w_{a}, & 2a-b>0, \vspace{1mm}\\
x_{b}x_a^{-1}(-1)^{p+q}\lambda B^{-2n-3+4a+2s}
     w_{2n+1}\otimes w_{a}, & 2a-b=0, \vspace{1mm}\\
x_{b}x_a^{-1}(-1)^{p+q}\lambda B^{-2n-3+2b+2s}
     w_{2n+1+2a-b}\otimes w_{a}, & 2a-b<0. 
\end{array}\right.    
\end{align*}

\par\noindent 
If $a<n+1$, $2\mid a$, $2\nmid b$, then 
\begin{align*}
&\quad \bar c(w_a\otimes w_b)=(\varphi_2^{-1}\otimes \varphi_1^{-1})c
        (x_{b}w_{2n+2-a}\otimes w_{2n+2-b})\\
&=x_{b}(-1)^qB^{2n+2s}(\varphi_2^{-1}\otimes \varphi_1^{-1})
     L^{2n+2-2a}(w_{2n+2-a}\otimes w_{2n+2-b})\\
&=\left\{\begin{array}{ll}
x_{b}(-1)^qB^{6n+4-4a+2s}(\varphi_2^{-1}\otimes \varphi_1^{-1})
      (w_{2a-b}\otimes w_a), &2a-b>0,\vspace{1mm}\\
x_{b}(-1)^{p+q}B^{6n+3-2b+2s}(\varphi_2^{-1}\otimes \varphi_1^{-1})
      (w_{b-2a+1}\otimes w_a), &2a-b<0,\vspace{1mm}\\
\end{array}\right. \\
&=\left\{\begin{array}{ll}
x_{b}x_a^{-1}(-1)^qB^{6n+4-4a+2s}
      w_{2a-b}\otimes w_a, &2a-b>0,\vspace{1mm}\\
x_{b}x_a^{-1}(-1)^{p+q}B^{6n+3-2b+2s}
      w_{2n+1+2a-b}\otimes w_a, &2a-b<0.\\
\end{array}\right.         
\end{align*} 

\par\noindent
If $a>n+1$, $2\mid a$, $2\mid b$, then     
\begin{align*}
&\quad \bar c(w_a\otimes w_b)=(\varphi_2^{-1}\otimes \varphi_1^{-1})c
        (x_{b}w_{2n+2-a}\otimes w_{b})\\
&=x_{b}(-1)^qB^{6n+4-4a+2s}(\varphi_2^{-1}\otimes \varphi_1^{-1})
     L^{-2n-2+2a}(w_{2n+2-a}\otimes w_{b})\\
&=\left\{\begin{array}{ll}
x_{b}(-1)^qB^{2n+2s}(\varphi_2^{-1}\otimes \varphi_1^{-1})
       (w_{2n+2-2a+b}\otimes w_{a}), & 2a-b<2n+2,\vspace{1mm}\\
x_{b}(-1)^{p+q}B^{6n+4-4a+2s+2b-1}(\varphi_2^{-1}\otimes \varphi_1^{-1})
        (w_{2a-b-2n-1}\otimes w_{a}), & 2a-b\geq 2n+2,\vspace{1mm}\\
\end{array}\right.   \\
&=\left\{\begin{array}{ll}
x_{b}x_a^{-1}(-1)^qB^{2n+2s}
       w_{2a-b}\otimes w_{a}, & 2a-b<2n+2,\vspace{1mm}\\
x_{b}x_a^{-1}(-1)^{p+q}B^{6n+3-4a+2b+2s}
        w_{2a-b-2n-1}\otimes w_{a}, & 2a-b\geq 2n+2.\vspace{1mm}\\
\end{array}\right.    
\end{align*}    

\par\noindent
If $a>n+1$, $2\mid a$, $2\nmid b$, then  
\begin{align*}
&\quad \bar c(w_a\otimes w_b)=(\varphi_2^{-1}\otimes \varphi_1^{-1})c
        (x_{b}w_{2n+2-a}\otimes w_{2n+2-b})\\
&=x_{b}(-1)^qB^{6n+4-4a+2s}(\varphi_2^{-1}\otimes \varphi_1^{-1})
     R^{-2n-2+2a}(w_{2n+2-a}\otimes w_{2n+2-b})\\
&=\left\{\begin{array}{ll}
x_{b}(-1)^qB^{6n+4-4a+2s}(\varphi_2^{-1}\otimes \varphi_1^{-1})
        (w_{2a-b}\otimes w_a), & 2a-b\leq 2n+1,\vspace{1mm}\\
x_{b}(-1)^{p+q}\lambda B^{2n+1-2b+2s}(\varphi_2^{-1}\otimes \varphi_1^{-1})
        (w_{4n+3-2a+b}\otimes w_a), & 2a-b\geq 2n+3,\vspace{1mm}\\
\end{array}\right.   \\
&=\left\{\begin{array}{ll}
x_{b}x_a^{-1}(-1)^qB^{6n+4-4a+2s}
        w_{2a-b}\otimes w_a, & 2a-b\leq 2n+1,\vspace{1mm}\\
x_{b}x_a^{-1}(-1)^{p+q}\lambda B^{2n+1-2b+2s}
        w_{2a-b-2n-1}\otimes w_a, & 2a-b\geq 2n+3.\vspace{1mm}\\
\end{array}\right.    
\end{align*}   

\par\noindent 
If $a>n+1$, $2\nmid a$, $2\nmid b$, then  
\begin{align*}
&\quad \bar c(w_a\otimes w_b)=(\varphi_2^{-1}\otimes \varphi_1^{-1})c
        (x_{b}w_{a}\otimes w_{2n+2-b})\\
&=x_{b}(-1)^qB^{2n+2s}(\varphi_2^{-1}\otimes \varphi_1^{-1})
     R^{-2n-2+2a}(w_{a}\otimes w_{2n+2-b})\\
&=\left\{\begin{array}{ll}
x_{b}(-1)^qB^{2n+2s}(\varphi_2^{-1}\otimes \varphi_1^{-1})
        (w_{2a-b}\otimes w_{2n+2-a}), &2a-b\leq 2n+1,\vspace{1mm}\\
x_{b}(-1)^{p+q}\lambda B^{-2n-3+4a-2b+2s}(\varphi_2^{-1}\otimes \varphi_1^{-1})
        (w_{4n+3-2a+b}\otimes w_{2n+2-a}), &2a-b\geq 2n+3,\vspace{1mm}\\
\end{array}\right.   \\
&=\left\{\begin{array}{ll}
x_{b}x_a^{-1}(-1)^qB^{2n+2s}
        w_{2a-b}\otimes w_{a}, &2a-b\leq 2n+1,\vspace{1mm}\\
x_{b}x_a^{-1}(-1)^{p+q}\lambda B^{-2n-3+4a-2b+2s}
        w_{2a-b-2n-1}\otimes w_{a}, &2a-b\geq 2n+3.\vspace{1mm}\\
\end{array}\right.    
\end{align*}   

\par\noindent
If $a>n+1$, $2\nmid a$, $2\mid b$, then   
\begin{align*}
&\quad \bar c(w_a\otimes w_b)=(\varphi_2^{-1}\otimes \varphi_1^{-1})c
        (x_{b}w_{a}\otimes w_{b})\\
&=x_{b}(-1)^qB^{2n+2s}(\varphi_2^{-1}\otimes \varphi_1^{-1})
     L^{-2n-2+2a}(w_{a}\otimes w_{b})\\
&=\left\{\begin{array}{ll}
x_{b}(-1)^qB^{-2n-4+4a+2s}(\varphi_2^{-1}\otimes \varphi_1^{-1})
         (w_{2n+2-2a+b}\otimes w_{2n+2-a}), & 2a-b<2n+2, \vspace{1mm}\\
x_{b}(-1)^{p+q}B^{2n-1+2b+2s}(\varphi_2^{-1}\otimes \varphi_1^{-1})
         (w_{2a-b-2n-1}\otimes w_{2n+2-a}), & 2a-b\geq 2n+2, \vspace{1mm}\\
\end{array}\right.   \\
&=\left\{\begin{array}{ll}
x_{b}x_a^{-1}(-1)^qB^{-2n-4+4a+2s}
         w_{2a-b}\otimes w_{a}, & 2a-b<2n+2, \vspace{1mm}\\
x_{b}x_a^{-1}(-1)^{p+q}B^{2n-1+2b+2s}
         w_{2a-b-2n-1}\otimes w_{a}, & 2a-b\geq 2n+2. \vspace{1mm}\\
\end{array}\right.    
\end{align*}     
\end{proof}

\subsection{The Nichols algebras $\mathfrak{B}\left(\mathscr{L}_{k,pq}^s\right)$}

Let $\mathscr{L}_{k,pq}^s$ be the $(2n+1)$-dimensional simple Yetter-Drinfeld module over 
$A_{N\, 2n+1}^{\mu\lambda}$, see \cite[Appendix]{Shi2020odd}. 
\begin{lemma}\label{NicholsAlgL}
The Yetter-Drinfeld module $\mathscr{L}_{k,pq}^s$ is of dihedral rack type $\Bbb D_{2n+1}$.
\end{lemma}
\begin{proof}
The Yetter-Drinfeld module  $\mathscr{L}_{k,pq}^s$ is of rack type, see  \cite[Section 4.3]{Shi2020odd}.
Let $X=\overline{1, 2n+1}$, and $r: X\times X\to X\times X$ satisfy the following conditions.  
\begin{enumerate}
\item If $2\nmid (a+b)$ for $a, b\in X$, set 
$b+(2a-1)=d_1(2n+1)+d_0$, $d_1\in\Bbb N$,  $d_0\in\overline{0, 2n}$.
Then 
\begin{align*}
r(a, b)=\left\{\begin{array}{ll}
(b+2a-1, a), & d_0=0,\\
(2n+1, a), & d_1=1, d_0=0,\\
(2n+2-d_0, a), &d_1=1, d_0\neq 0,\\
(1, a), & d_1=2, d_0=0,\\
(d_0, a), & d_1=2, d_0\neq 0.
\end{array}\right.
\end{align*} 
\item If $2\mid (a+b)$ for $a, b\in X$, set $2a-1=b+d_1(2n+1)+d_0$, $d_1\in\Bbb N$, 
$d_0\in\overline{0, 2n}$. Then 
\begin{align*}
r(a, b)=\left\{\begin{array}{ll}
(b-2a+1, a), & 2a-1<b,\\
(2a-b, a), &d_1=0,\\
(2n+1-d_0, a), &d_1=1. 
\end{array}\right. 
\end{align*}
\end{enumerate}
According to \cite[Section 4.3]{Shi2020odd}, $\mathscr{L}_{k,pq}^s$ is associated to 
the set-theoretical solution $(X, r)$ of the Yang-Baxter equation defined above. 
So we only need to prove that $(X, r)$ is isomorphic to the solution of the Yang-Baxter equation 
arising from the dihedral rack $\Bbb D_{2n+1}$. 

Let $f: X\to X$ such that $f(a)=\left\{\begin{array}{ll}
a, & a \text{ odd}, \\
2n+2-a, & a \text{ even}, \\
\end{array}\right.$
for all $a\in X$. 
Suppose $(f^{-1}\times f^{-1})r(f\times f)(a, b)=(\gamma, a)$, then we claim that 
\begin{align}\label{DihedralRelation}
\gamma\equiv 2a-b \mod (2n+1).
\end{align} 
The formula \eqref{DihedralRelation} can be proved via case by case verification. For example, 
if $2\mid a$, $2\mid b$, set
$2(2n+2-a)-1=(2n+2-b)+d_1(2n+1)+d_0$,   $d_1\in\Bbb N$,
$d_0\in\overline{0, 2n}$. Then we have
\begin{align*}
&\quad (f^{-1}\times f^{-1})r(f\times f)(a, b)=(f^{-1}\times f^{-1})r(2n+2-a, 2n+2-b)\\
&=\left\{\begin{array}{ll}
(f^{-1}\times f^{-1})(2a-b-2n-1, 2n+2-a), & 2a-b>2n+1,\\
(f^{-1}\times f^{-1})(2n+2-2a+b, 2n+2-a), & 2a-b\in\overline{1, 2n+1},\\
(f^{-1}\times f^{-1})(2n+1+2a-b, 2n+2-a), & 2a-b\leq 0,\\
\end{array}\right.\\
&=\left\{\begin{array}{ll}
(2a-b-2n-1, a), & 2a-b>2n+1,\\
(2a-b, a), & 2a-b\in\overline{1, 2n+1},\\
(2n+1+2a-b, a), & 2a-b\leq 0.
\end{array}\right.
\end{align*}
\end{proof}

\subsection{The Nichols algebra $\mathfrak{B}\left(\mathscr{I}_{pjk}^s\right)$}
Let $\mathscr{I}_{pjk}^s=\bigoplus_{a=1}^{n}\left(\k w_a\oplus\k m_a\right)$ be  the $2n$-dimensional 
simple Yetter-Drinfeld module over $A_{N\,2n}^{\mu\lambda}$, see \cite[Appendix]{Shi2020even}. 
Set $w_{n+a}=m_a$ for $a\in\overline{1, n}$. It is not difficult to see the braided vector space
$\mathscr{I}_{pjk}^s$ is associated to a set-theoretical solution $(X, r)$ of the Yang-Baxter equation, where 
$X=\overline{1, 2n}$, and $r$ is described as follows. 
\begin{enumerate}
\item If $2\mid (a+b)$, $a, b\in\overline{1, n}$, then 
\begin{align*}
r(a, b)=\left\{\begin{array}{ll}
(1, a), &b=2a,\\
(b-2a+1, a), &b>2a,\\
(2a-b, a), & 0<2a-b\leq n,\\
(2n+1-2a+b, a), & 1+n\leq 2a-b.\\
\end{array}\right.
\end{align*}

\item If $2\nmid (a+b)$, $a, b\in\overline{1, n}$, then
\begin{align*}
r(a, b)=\left\{\begin{array}{ll}
(2a+b-1, a), &2a+b-1\leq n,\\
(2n+2-2a-b, a), &n\leq 2a+b-2\leq 2n-1,\\
(2a+b-2n-1, a), & 2n\leq 2a+b-2.\\
\end{array}\right.
\end{align*}

\item If $2\nmid (a+b)$, $a, b\in\overline{n+1, 2n}$, then 
\begin{align*}
r(a, b)=\left\{\begin{array}{ll}
(2a+b-2n-1, a), &3n+3\leq 2a+b\leq 4n+1,\\
(6n+2-2a-b, a), &4n+2\leq 2a+b\leq 5n+1,\\
(2a+b-4n-1, a), & 5n+2\leq 2a+b.\\
\end{array}\right.
\end{align*}

\item If $2\mid (a+b)$, $a, b\in\overline{n+1, 2n}$, then 
\begin{align*}
r(a, b)=\left\{\begin{array}{ll}
(n+1, a), & 2a-b=n,\\
(2a-b, a), & n<2a-b\leq 2n,\\
(4n+1-2a+b, a), & 2n+1\leq 2a-b\leq 3n-1,\\
(2n+1+b-2a, a), & 0<2a-b<n.
\end{array}\right. 
\end{align*}

\item If $2\mid (a+b)$, $a\in\overline{n+1, 2n}$, $b\in\overline{1,n}$, then 
\begin{align*}
r(a, b)=\left\{\begin{array}{ll}
(2a+b-2n-1, a), & 2n+3\leq 2a+b\leq 3n+1,\\
(4n+2-2a-b, a), & 3n+2\leq 2a+b\leq 4n+1,\\
(2a+b-4n-1, a), & 4n+2\leq 2a+b\leq 5n.\\
\end{array}\right. 
\end{align*}

\item If $2\nmid (a+b)$, $a\in\overline{n+1, 2n}$, $b\in\overline{1,n}$, then
\begin{align*}
r(a, b)=\left\{\begin{array}{ll}
(1, a), & 2a-b=2n,\\
(2a-b-2n, a), & 2n< 2a-b\leq 3n,\\
(4n+1-2a+b, a), & 3n+1\leq 2a-b< 4n,\\
(b-2a+2n+1, a), & n+1\leq 2a-b< 2n.\\
\end{array}\right. 
\end{align*}

\item If $2\nmid (a+b)$, $a\in\overline{1, n}$, $b\in\overline{n+1, 2n}$, then
\begin{align*}
r(a, b)=\left\{\begin{array}{ll}
(n+1, a), & 2a-b=-n,\\
(2n+2a-b, a), & -n<2a-b\leq 0,\\
(2n+1-2a+b, a), & 1\leq 2a-b, \\
(1+b-2a, a), & 2a-b<-n.
\end{array}\right. 
\end{align*}

\item If $2\mid (a+b)$, $a\in\overline{1, n}$, $b\in\overline{n+1, 2n}$, then
\begin{align*}
r(a, b)=\left\{\begin{array}{ll}
(2a+b-1, a), & 2a+b\leq 2n+1,\\
(4n+2-2a-b, a), & 2n+2\leq 2a+b\leq 3n+1, \\
(2a+b-2n-1, a), & 2a+b\geq 3n+2.
\end{array}\right. 
\end{align*}
\end{enumerate}

\begin{lemma}\label{NicholsAlgI}
The Yetter-Drinfeld module $\mathscr{I}_{pjk}^s$ is of dihedral rack type $\Bbb D_{2n}$.
\end{lemma}
\begin{proof}
We only need to prove that the set-theoretical solution $(X, r)$ of the Yang-Baxter equation 
is  of dihedral rack type $\Bbb D_{2n}$. That is to say, we need to find a bijective map 
$f: X\to X$, such that 
\begin{align}\label{DihdedraF}
(f^{-1}\times f^{-1})r(f\times f)(a, b)=(\gamma, a), \quad 
\gamma\equiv 2a-b\mod 2n,\quad \forall a, b\in X.
\end{align}
Define the map $f: X\to X$ as 
\[
f(a)=\left\{\begin{array}{ll}
2n+1-a, & a\text{ even}, 1<a\leq n, \\
2n+1-a, & a\text{ odd},  n<a<2n, \\
a, & otherwise. 
\end{array}\right.
\]
The formula \eqref{DihdedraF} can be verified case by case. For example, if 
$2\mid a$, $2\mid b$, and $a, b\in\overline{1, n}$, then 
\begin{align*}
&\quad (f^{-1}\times f^{-1})r(f\times f)(a, b)
=(f^{-1}\times f^{-1})r(2n+1-a, 2n+1-b)\\
&=\left\{\begin{array}{ll}
(f^{-1}\times f^{-1})(n+1, 2n+1-a), & 2a-b=n+1,\\
(f^{-1}\times f^{-1})(2n+1-2a+b, 2n+1-a), & 1\leq 2a-b<n+1,\\
(f^{-1}\times f^{-1})(2n+2a-b, 2n+1-a), & 2-n\leq 2a-b\leq 0,\\
(f^{-1}\times f^{-1})(2a-b, 2n+1-a), & n+1< 2a-b<2n+1,\\
\end{array}\right.\\
&=\left\{\begin{array}{ll}
(n+1, a), & 2a-b=n+1,\\
(2a-b, a), & 1\leq 2a-b<n+1,\\
(2n+2a-b, a), & 2-n\leq 2a-b\leq 0,\\
(2a-b, a), & n+1< 2a-b<2n+1.\\
\end{array}\right.
\end{align*}
\end{proof}


\begin{thebibliography}{10}
\expandafter\ifx\csname urlstyle\endcsname\relax
  \providecommand{\doi}[1]{doi:\discretionary{}{}{}#1}\else
  \providecommand{\doi}{doi:\discretionary{}{}{}\begingroup
  \urlstyle{rm}\Url}\fi

\bibitem{Andruskiewitsch2017}
Andruskiewitsch, N., Angiono, I. (2017).
\newblock On finite dimensional {N}ichols algebras of diagonal type.
\newblock \emph{Bull. Math. Sci.} 7(3):353--573.
\newblock \doi{10.1007/s13373-017-0113-x}.

\bibitem{andruskiewitsch2007pointed}
Andruskiewitsch, N., Fantino, F. (2007).
\newblock On pointed {H}opf algebras associated with alternating and dihedral
  groups.
\newblock \emph{Rev. Un. Mat. Argentina} 48(3):57--71 (2008).

\bibitem{Andruskiewitsch2010}
Andruskiewitsch, N., Fantino, F., Garc\'{\i}a, G.~A., Vendramin, L. (2010).
\newblock On twisted homogeneous racks of type {D}.
\newblock \emph{Rev. Un. Mat. Argentina} 51(2):1--16.

\bibitem{Andruskiewitsch2011}
Andruskiewitsch, N., Fantino, F., Gra\~{n}a, M., Vendramin, L. (2011).
\newblock Finite-dimensional pointed {H}opf algebras with alternating groups
  are trivial.
\newblock \emph{Ann. Mat. Pura Appl. (4)} 190(2):225--245.
\newblock \doi{10.1007/s10231-010-0147-0}.

\bibitem{Andruskiewitsch2018}
Andruskiewitsch, N., Giraldi, J. M.~J. (2018).
\newblock Nichols algebras that are quantum planes.
\newblock \emph{Linear and Multilinear Algebra} 66(5):961--991.
\newblock \doi{10.1080/03081087.2017.1331997}.

\bibitem{Andruskiewitsch2003MR1994219}
Andruskiewitsch, N., Gra{\~n}a, M. (2003).
\newblock From racks to pointed {H}opf algebras.
\newblock \emph{Adv. Math.} 178(2):177--243.
\newblock \doi{10.1016/S0001-8708(02)00071-3}.

\bibitem{andruskiewitsch2001pointed}
Andruskiewitsch, N., Schneider, H.-J. (2002).
\newblock Pointed {H}opf algebras.
\newblock In: \emph{New directions in {H}opf algebras}, vol.~43 of \emph{Math.
  Sci. Res. Inst. Publ.}, pp. 1--68. Cambridge Univ. Press, Cambridge.
\newblock \doi{10.2977/prims/1199403805}.

\bibitem{Brieskorn1988}
Brieskorn, E. (1988).
\newblock Automorphic sets and braids and singularities.
\newblock Braids, {AMS}-{IMS}-{SIAM} {Jt}. {Summer} {Res}. {Conf}., {Santa}
  {Cruz}/{Calif}. 1986, {Contemp}. {Math}. 78, 45-115 (1988).

\bibitem{Dijkgraaf1991}
Dijkgraaf, R., Pasquier, V., Roche, P. (1991).
\newblock Quasi-quantum groups related to orbifold models.
\newblock Modern quantum field theory, {Proc}. {Int}. {Colloq}.,
  {Bombay}/{India} 1990, 375-383 (1991).

\bibitem{heckenberger2009classification}
Heckenberger, I. (2009).
\newblock Classification of arithmetic root systems.
\newblock \emph{Adv. Math.} 220(1):59--124.
\newblock \doi{10.1016/j.aim.2008.08.005}.

\bibitem{Heckenberger2017}
Heckenberger, I., Vendramin, L. (2017).
\newblock The classification of {N}ichols algebras over groups with finite root
  system of rank two.
\newblock \emph{J. Eur. Math. Soc. (JEMS)} 19(7):1977--2017.
\newblock \doi{10.4171/JEMS/711}.

\bibitem{Huang2018}
Huang, H.-L., Yang, Y., Zhang, Y. (2018).
\newblock On nondiagonal finite quasi-quantum groups over finite abelian
  groups.
\newblock \emph{Selecta Mathematica. New Series} 24(5):4197--4221.
\newblock \doi{10.1007/s00029-018-0420-4}.

\bibitem{MR0208401}
Kac, G.~I., Paljutkin, V.~G. (1966).
\newblock Finite ring groups.
\newblock \emph{Trudy Moskov. Mat. Ob\v s\v c.} 15:224--261.

\bibitem{LuJianghua2000MR1769723}
Lu, J.-H., Yan, M., Zhu, Y.-C. (2000).
\newblock On the set-theoretical {Y}ang-{B}axter equation.
\newblock \emph{Duke Math. J.} 104(1):1--18.
\newblock \doi{10.1215/S0012-7094-00-10411-5}.

\bibitem{Natale2003}
Natale, S. (2003).
\newblock On group theoretical {H}opf algebras and exact factorizations of
  finite groups.
\newblock \emph{J. Algebra} 270(1):199--211.
\newblock \doi{10.1016/S0021-8693(03)00464-2}.

\bibitem{Schauenburg1992}
Schauenburg, P. (1992).
\newblock \emph{On coquasitriangular {H}opf algebras and the quantum
  {Y}ang-{B}axter equation}, vol.~67 of \emph{Algebra Berichte [Algebra
  Reports]}.
\newblock Verlag Reinhard Fischer, Munich.

\bibitem{MR1396857}
Schauenburg, P. (1996).
\newblock A characterization of the {B}orel-like subalgebras of quantum
  enveloping algebras.
\newblock \emph{Comm. Algebra} 24(9):2811--2823.
\newblock \doi{10.1080/00927879608825714}.

\bibitem{Shi2023}
Shi, Y.-X. (2023).
\newblock Multinomial expansion and {Nichols} algebras associated to
  non-degenerate involutive solutions of the {Yang-Baxter} equation.
\newblock \emph{Comm. Algebra,} \doi{10.1080/00927872.2023.2165660}.

\bibitem{Shi2019}
Shi, Y.-X. (2019).
\newblock Finite-dimensional {H}opf algebras over the {K}ac-{P}aljutkin algebra
  {$H_8$}.
\newblock \emph{Rev. Un. Mat. Argentina} 60(1):265--298.
\newblock \doi{10.33044/revuma.v60n1a17}.

\bibitem{Shi2020even}
Shi, Y.-X. (2020).
\newblock Finite-dimensional {N}ichols algebras over the {S}uzuki algebras
  \uppercase\expandafter{\romannumeral1}: simple {Y}etter-{D}rinfeld modules of
  ${A}_{N\,2n}^{\mu\lambda}$.
\newblock \emph{arXiv:2011.14274, to appear in Bull. Belg. Math. Soc. Simon
  Stevin} .

\bibitem{Shi2020odd}
Shi, Y.-X. (2021).
\newblock Finite-dimensional {N}ichols algebras over the {S}uzuki algebras
  \uppercase\expandafter{\romannumeral2}: simple {Y}etter-{D}rinfeld modules of
  ${A}_{N\,2n+1}^{\mu\lambda}$.
\newblock \emph{arXiv:2103.06475} .

\bibitem{zbMATH01585085}
Soloviev, A. (2000).
\newblock Non-unitary set-theoretical solutions to the quantum {Yang}-{Baxter}
  equation.
\newblock \emph{Math. Res. Lett.} 7(5-6):577--596.
\newblock \doi{10.4310/MRL.2000.v7.n5.a4}.

\bibitem{Suzuki1998}
Suzuki, S. (1998).
\newblock A family of braided cosemisimple {H}opf algebras of finite dimension.
\newblock \emph{Tsukuba J. Math.} 22(1):1--29.
\newblock \doi{10.21099/tkbjm/1496163467}.

\end{thebibliography}
\end{document}